\documentclass[onefignum,onetabnum,nohypdvips]{siamart251216}
\usepackage{placeins}
\usepackage{booktabs}
\usepackage{array}
\usepackage{tabularx}

\usepackage{amsfonts}
\usepackage{amsmath}
\usepackage{amssymb}
\usepackage{graphicx}
\usepackage{tikz}
\usepackage{algorithmic}
\usepackage{hyperref}

\ifpdf
  \DeclareGraphicsExtensions{.eps,.pdf,.png,.jpg}
\else
  \DeclareGraphicsExtensions{.eps}
\fi

\newsiamremark{remark}{Remark}

\headers{Rank Compensation for Matrices and Tensor Trains}{I. El Mountasser and K. Jbilou}

\title{Low-Precision Rank Compensation for Matrices and Tensor Trains}

\author{
Ibrahim El Mountasser\thanks{
Université du Littoral Côte d'Opale, LMPA,
F-62100 Calais, France
(\email{ibrahimelmountasser@hotmail.com},
\email{jbilou@univ-littoral.fr}).
}
\and
Khalide Jbilou\footnotemark[1]
}

\usepackage{amsopn}

\DeclareMathOperator{\fl}{fl}
\newcommand{\T}{\mathsf{T}}

\makeatletter
\providecommand{\headerps@out}[1]{}
\@ifundefined{pdf@box}{\newsavebox{\pdf@box}}{}
\providecommand{\literalps@out}[1]{}
\providecommand{\pdf@addtoksx}[1]{}
\makeatother

\hypersetup{
  pdftitle={Low-Precision Rank Compensation for Matrices and Tensor Trains},
  pdfauthor={Ibrahim El Mountasser and Khalide Jbilou}
}

\begin{document}

\maketitle

\begin{abstract}
Lower numerical precision reduces storage and memory traffic but raises the perturbation floor. We study rank compensation: reinvesting saved memory in a larger approximation rank. For matrices, the singular-value error identity yields a directly testable sufficient condition requiring the additional singular component to offset the perturbation from storing the rank-augmented approximation in lower precision. On ten SuiteSparse matrices, all 100 truncation-dominated configurations (50 FP32 and 50 FP16) are certified non-increases and strict accuracy wins, with mean error ratio $0.963$ and storage ratios $58.8\%$ and $29.4\%$ relative to the FP64 baseline. FP16 failures occur only in tail-rank stress tests near the perturbation floor. At the largest resident matrix-application batch, compensated FP32 and FP16 achieve geometric-mean A100 speedups of $1.28\times$ and $2.12\times$; neither accelerates the smallest batch. For Tensor-Train (TT) approximation, we give a conditional a posteriori extension based on the measured truncation gain and rounded-core perturbation. Across three-way and six-way synthetic tests, FP32 and FP16 achieve combined accuracy--memory wins in 10 of 20 and 14 of 20 trials. On public hyperspectral tensors and FROSTT top-active subtensors, the corresponding counts are 44 of 60 and 54 of 60; four FP16 Salinas-A tail-stress cases fail. No certified TT case exceeds the FP64 error beyond numerical tolerance. Reconstruction of six public tensors yields geometric-mean compensated speedups of $1.38\times$ (FP32) and $1.94\times$ (FP16). Timings cover resident downstream kernels, not factorization, transfers, or end-to-end acceleration.
\end{abstract}


\section{Introduction} \label{sec:introduction}

Low-rank approximation is one of the main tools for reducing storage and arithmetic cost in numerical linear algebra.  A second, increasingly important, tool is lower-precision arithmetic.  Moving from FP64 to FP32 or FP16 can reduce storage and memory traffic and can expose faster hardware units, but the larger unit roundoff increases the perturbation floor.  A natural question is whether part of the memory saved by lower precision can be reinvested in a slightly larger rank so that the final approximation is no worse, and sometimes better, than a higher-precision lower-rank baseline.

We first address this question in the matrix setting.  The Eckart--Young low-rank approximation theorem gives an exact expression for the Frobenius-norm error of the best rank-$k$ approximation~\cite{EckartYoung1936,GolubVanLoan2013}, so the effect of moving from rank $k$ to rank $k+1$ is transparent: the squared truncation error is reduced by $\sigma_{k+1}^2$.  If the low-precision realization of the rank-$(k+1)$ approximation introduces a perturbation of size $\eta$, then the triangle inequality gives a conservative and directly testable sufficient condition for the rank increase to compensate for the lower precision.

The matrix result serves as a model problem for the tensor case.  Tensor--Train (TT) decomposition offers large storage reductions for high-dimensional arrays \cite{Oseledets2011}, but the same rank-compensation idea becomes more delicate because TT rank increments are coupled across unfoldings and low-precision rounding acts on nonunique cores.  We therefore use the matrix analysis to isolate the central memory--accuracy principle before formulating the TT extension and the additional assumptions it requires.

We derive a sufficient matrix certificate coupling a precision reduction to a
rank increment and a conditional a posteriori extension for TT representations.
Experiments on synthetic and public data evaluate accuracy, storage, and
downstream A100 performance.

\Cref{sec:related_work} reviews related work, and \cref{sec:floating_point}
introduces the numerical formats and timing protocol. The matrix analysis and
experiments appear in \cref{sec:matrix_rank_precision}.
\Cref{sec:tt_preliminaries,sec:theoretical_analysis,sec:tt_experiments} develop
and test the TT extension, followed by the fixed-memory application in
\cref{sec:hyperspectral_application}. \Cref{sec:conclusion} concludes.

\section{Related Work} \label{sec:related_work}

Mixed-precision techniques for numerical linear algebra have become an active area of research, driven by the performance differences between numerical formats on modern hardware; see Higham and Mary~\cite{HighamMary2022} and Abdelfattah et al.~\cite{AbdelfattahEtAl2021} for surveys. Much of this literature focuses on accelerating linear solvers and matrix factorizations by combining low-precision kernels with higher-precision correction or refinement, for example in three-precision iterative refinement~\cite{CarsonHigham2018}.

Recent work also treats low-rank approximation directly. Amestoy et
al.~\cite{AmestoyEtAl2023} assign precisions within low-rank blocks
according to singular-vector contributions and apply the representation to
block low-rank LU factorization. Carson and Dau\v{z}ickait\.e
\cite{CarsonDauzickaite2024} analyze a two-precision single-pass Nystr\"om
approximation and give a heuristic for selecting the lower precision. Baboulin
et al.~\cite{BaboulinEtAl2024GPU} use GPU Tensor Cores in randomized low-rank
approximation, and their later iterative-refinement framework
\cite{BaboulinEtAl2025} covers matrix and tensor decompositions, including TT.

Our focus is narrower: we analyze the memory--accuracy trade-off of a stored
compressed representation when precision savings are reinvested in rank. The
Eckart--Young theorem is classical~\cite{EckartYoung1936,GolubVanLoan2013}, as
are unitarily invariant norm extensions~\cite{Mirsky1960} and the floating-point
tools used here~\cite{Higham2002}. To the best of our knowledge, prior work has
not coupled an explicit memory-derived rank budget to a directly measurable
sufficient certificate comparing truncation gain with storage-quantization
perturbation, either for matrices or for TT approximation. The TT case requires
additional care because truncations are coupled and rounding acts on
gauge-dependent cores~\cite{Oseledets2011,HoltzRohwedderSchneider2012,Schollwock2011}.

\section{Floating-Point Arithmetic and Hardware} \label{sec:floating_point}

\subsection{Low-Precision Motivation} \label{subsubsec:motivation_low_precision}
Storing low-rank factors and TT cores in FP32 or FP16 reduces memory use and
traffic relative to FP64~\cite{HighamMary2022}. Lower precision can also provide
higher throughput on accelerators; A100 Tensor Cores support TF32, FP16, BF16,
and selected FP64 operations~\cite{NVIDIAA1002020,AbdelfattahEtAl2021}.
Reduced data movement may lower energy use, although energy is not measured
here~\cite{AbdelfattahEtAl2021,HighamMary2022}. These benefits motivate the
precision--rank trade-off studied below.

\subsection{Common Numerical Formats and NVIDIA A100 Performance} \label{subsubsec:numerical_formats}
The IEEE 754 standard specifies FP32 and FP64 arithmetic~\cite{IEEE7542019}.
TF32 and BF16 are additional accelerator formats documented in the NVIDIA A100
white paper~\cite{NVIDIAA1002020}. \Cref{tab:precision_comparison_prelim} lists
their bit allocations, unit roundoff, and A100 peak rates. Under round-to-nearest,
a binary format with $p$ significand bits has
$u=2^{-p}$~\cite{Higham2002,IEEE7542019}.

\begin{table}[!htbp]
    \centering
    \caption{Comparison of selected numerical formats. Significand bits include
    the implicit leading bit for normalized floating-point numbers. Unit
    roundoff is approximated by $2^{-p}$. A100 rates are peak dense TFLOP/s
    (TOPS for INT8); TC denotes Tensor Cores~\cite{NVIDIAA1002020}.}
    \label{tab:precision_comparison_prelim}
    \setlength{\tabcolsep}{3.5pt} 
    \begin{tabularx}{\linewidth}{@{} >{\raggedright\arraybackslash}X ccccc r @{}}
        \toprule
        Format & \shortstack{Total\\Bits} & \shortstack{Sign\\Bits} & \shortstack{Exp.\\Bits} & \shortstack{Sig.\\(Total)} & $u \approx$ & \shortstack{A100 Peak\\Perf.} \\
        \midrule
        FP64 (double)   & 64    & 1     & 11    & 53          & $10^{-16}$ & 9.7 \\
        FP64 (TC)       & 64    & 1     & 11    & 53          & $10^{-16}$ & 19.5 \\
        FP32 (single)   & 32    & 1     & 8     & 24          & $10^{-7}$  & 19.5 \\
        TF32 (TC)\textsuperscript{a}& (32) & 1 & 8 & 11          & $5 \times 10^{-4}$ & 156 \\
        BF16 (TC)       & 16    & 1     & 8     & 8           & $4 \times 10^{-3}$ & 312 \\
        FP16 (TC)       & 16    & 1     & 5     & 11          & $5 \times 10^{-4}$ & 312 \\
        INT8 (TC)       & 8     & varies& 0     & up to 8     & N/A        & 624 \\
        \bottomrule
    \end{tabularx}
    \vspace{2mm}
    \textsuperscript{a}\footnotesize{NVIDIA's TensorFloat-32 (TF32) uses FP32 storage and an 8-bit exponent. Tensor Core input conversion retains 10 fraction bits, giving 11 effective significand bits when the implicit leading bit is included and matching the convention in the table \cite{NVIDIAA1002020}. Peak performance can reach 312 TFLOP/s with sparsity when applicable.}
\end{table}

FP16 has a narrower exponent range than FP32 and BF16 and may require careful
scaling to avoid overflow or underflow. BF16 retains the FP32 exponent width but
has fewer significand bits than FP16~\cite{IEEE7542019,NVIDIAA1002020}.

\subsection{Floating-Point Error Model}
\label{subsubsec:rounding_model_constant_c}

We adopt the standard backward--error model for floating--point arithmetic
\cite{Higham2002}, in which each elementary operation is performed with a
relative error bounded by the \emph{unit roundoff}~$u$ of the working precision,
provided the result neither overflows nor underflows:
\begin{equation}\label{eq:fl_op_model}
  \operatorname{fl}\bigl(x\,\text{op}\,y\bigr)
  \;=\;
  \bigl(x\,\text{op}\,y\bigr)(1+\epsilon),
  \qquad
  |\epsilon| \le u.
\end{equation}
Values of~$u$ for the floating-point formats used below appear in
Table~\ref{tab:precision_comparison_prelim}.

\subsection{GPU Benchmark Environment and Timing Protocol}
\label{sec:gpu-protocol}

The matrix and TT benchmarks share one protocol. Their configurations appear in
\cref{sec:matrix-performance} and \cref{sec:tt-performance}, respectively. Both
ran on the NVIDIA A100 configuration in \cref{tab:gpu-environment}.
For every method, its input, factors or cores, workspaces, and output use the
method precision and remain resident on the GPU. Only the downstream kernel is
timed. Excluded from timing are host-to-device transfers and allocations; data
loading and validation; SVD or TT-SVD; rank selection; orthogonalization; and
quantization. Thus the measurements are neither transfer-inclusive nor
end-to-end latency.

TF32 paths, reduced-precision FP16 reductions, and explicit FP16 accumulation
are disabled, leaving FP32 accumulation for FP16 matrix products. Each task
begins with 30 untimed warmups. A seven-event calibration then selects the
number of inner repetitions needed to place each recorded CUDA event near
2~ms, subject to a cap of 256 repetitions; event times are divided by that
count. The five methods are independently randomized within every complete
block. Each task comprises three rounds of five blocks and 20 event samples per
method in each block, giving 300 samples and 15 paired block medians per task.
The task speedup is the median of the 15 ratios between matching FP64 and method
block medians. Dataset aggregates use geometric means, and their 95\% intervals
are percentile intervals from 10,000 case-level bootstrap resamples.

For storage and approximation-error ratios, the rounded representations are
contracted in FP64. Separate untimed output checks compare matrix output slices with
FP64 evaluation of the stored factors and sampled TT entries with an independent
core-chain evaluator. The maximum observed relative discrepancies from executing
the kernel at the method precision are $3.5\times10^{-7}$ for FP32 and
$4.5\times10^{-4}$ for FP16.

\begin{table}[!htbp]
  \centering
  \caption{Hardware, software, and timing protocol shared by the
  matrix-application and TT-reconstruction benchmarks.}
  \label{tab:gpu-environment}
  \resizebox{\linewidth}{!}{\begin{tabular}{ll}
\toprule
Item & Value \\
\midrule
GPU & NVIDIA A100-SXM4-80GB (80 GiB) \\
Driver / CUDA runtime & 580.126.20 / 12.8 \\
PyTorch / cuDNN & 2.8.0+cu128 / 9.10.2 \\
TF32 / FP16 reduced reduction & disabled / disabled \\
Warmup / timing design & 30 / 3 rounds $\times$ 5 blocks $\times$ 20 samples \\
Samples per task / tasks & 300 / 190 \\
Timing / residency & CUDA events / inputs and workspaces resident \\
Primary timing excludes & data loading, factorization, allocation, and transfer \\
\bottomrule
\end{tabular}
}
\end{table}

\section{Matrix Rank Compensation} \label{sec:matrix_rank_precision}

Let $A\in\mathbb{R}^{m\times n}$ have singular value decomposition
\[
  A = U\Sigma V^{\T}, \qquad
  \sigma_1\ge \sigma_2\ge \cdots \ge 0.
\]
Denote by
\[
  A_k = \sum_{j=1}^k \sigma_j u_jv_j^{\T}
\]
the best rank-$k$ approximation in the Frobenius norm~\cite{EckartYoung1936,GolubVanLoan2013}, and write
\[
  E_k=\|A-A_k\|_F
  =
  \left(\sum_{j>k}\sigma_j^2\right)^{1/2}.
\]
The central question is whether a low-precision realization of $A_{k+1}$ can improve upon a high-precision realization of $A_k$ while using less memory.

\subsection{Rank Compensation Bound}

The following condition follows from the Eckart--Young identity and the triangle
inequality~\cite{EckartYoung1936,GolubVanLoan2013}.

\begin{theorem}[Matrix rank compensation]\label{thm:matrix_rank_compensation}
Let $\widehat A_{k+1}$ be any low-precision realization of $A_{k+1}$ and define
\[
  \eta_{k+1}=\|A_{k+1}-\widehat A_{k+1}\|_F .
\]
Then
\begin{equation}\label{eq:matrix_universal_bound}
  \|A-\widehat A_{k+1}\|_F
  \le
  \bigl(E_k^2-\sigma_{k+1}^2\bigr)^{1/2}
  +\eta_{k+1}.
\end{equation}
In particular, if
\begin{equation}\label{eq:matrix_condition}
  E_{k+1}+\eta_{k+1}\le E_k,
\end{equation}
then
\[
  \|A-\widehat A_{k+1}\|_F \le E_k.
\]
In exact arithmetic, \eqref{eq:matrix_condition} is equivalent to the
conjunction of $\eta_{k+1}\le E_k$ and
\begin{equation}\label{eq:matrix_condition_squared}
  \sigma_{k+1}^2 \ge 2\eta_{k+1}E_k-\eta_{k+1}^2.
\end{equation}
\end{theorem}

\begin{proof}
By the singular-value error formula,
\[
  \|A-A_{k+1}\|_F^2
  =
  \sum_{j>k+1}\sigma_j^2
  =
  E_k^2-\sigma_{k+1}^2 .
\]
The triangle inequality gives
\[
  \|A-\widehat A_{k+1}\|_F
  \le
  \|A-A_{k+1}\|_F+\|A_{k+1}-\widehat A_{k+1}\|_F,
\]
which proves \eqref{eq:matrix_universal_bound}.  Since
$E_{k+1}=(E_k^2-\sigma_{k+1}^2)^{1/2}$,
\eqref{eq:matrix_condition} makes the right-hand side of
\eqref{eq:matrix_universal_bound} at most $E_k$.  Moreover,
\eqref{eq:matrix_condition} implies $\eta_{k+1}\le E_k$; after moving
$\eta_{k+1}$ to the right and squaring, one obtains
\eqref{eq:matrix_condition_squared}.  Reversing these steps proves the stated
equivalence.
\end{proof}

\noindent In numerical checks we evaluate \eqref{eq:matrix_condition} directly.
The squared form \eqref{eq:matrix_condition_squared} is retained for algebraic
interpretation and is used only together with the required condition
\(\eta_{k+1}\le E_k\).

\begin{remark}
The theorem is independent of the particular low-precision mechanism.  In the experiments below, $\widehat A_{k+1}$ is formed by rounding the stored factors $U_{k+1}$, $\Sigma_{k+1}$, and $V_{k+1}^{\T}$ to FP32 or FP16 and then reconstructing in FP64 for error measurement.  Thus the reported $\eta_{k+1}$ isolates storage quantization; it does not include error from evaluating the factor product in low-precision arithmetic.  One could instead round the dense reconstructed matrix or include low-precision evaluation error in $\widehat A_{k+1}$; the same bound applies with the corresponding value of $\eta_{k+1}$.
\end{remark}

We use the same terminology throughout the numerical sections. An
\emph{accuracy non-increase} means that the low-precision augmented
approximation has Frobenius error no larger than the FP64 baseline, whereas an
\emph{accuracy win} means a strict decrease. A \emph{memory win} means that its
exact stored representation is smaller than the FP64 baseline, and a
\emph{practical win} means both strict accuracy and memory wins. A
\emph{certified non-increase} means that the sufficient condition proves the
nonstrict accuracy statement. Certification alone proves neither a strict
accuracy win nor a memory win. The reported empirical wins use strict
comparisons; numerical tolerances are used only for certificate boundaries and
algebraic identities.

\subsection{Storage Trade-Off}

If a rank-$k$ SVD representation is stored through its factors, the leading storage count is proportional to
\[
  b\,k(m+n+1),
\]
where $b$ is the number of bits per stored scalar.  Replacing an FP64 rank-$k$ representation by a lower-precision rank-$(k+1)$ representation therefore gives the idealized storage ratio
\begin{equation}\label{eq:matrix_storage_ratio}
  \frac{\mathrm{storage}_{\mathrm{low}}(k+1)}
       {\mathrm{storage}_{\mathrm{FP64}}(k)}
  =
  \frac{b_{\mathrm{low}}}{64}\frac{k+1}{k}.
\end{equation}
Thus FP32 still saves memory when $k>1$, and FP16 saves memory for all positive ranks.  The numerical question is whether the augmented tail plus the storage perturbation satisfies \(E_{k+1}+\eta_{k+1}\le E_k\), as required by \cref{thm:matrix_rank_compensation}.

\subsection{Public SuiteSparse Matrix Experiments}\label{subsec:matrix_experiments}

We tested the bound on ten matrices from the SuiteSparse Matrix Collection \cite{DavisHu2011,KolodziejEtAl2019}, all from the original Harwell--Boeing (\texttt{HB}) group~\cite{DuffGrimesLewis1989}: \texttt{ash85}, \texttt{ash219}, \texttt{ash331}, \texttt{ash608}, \texttt{494\_bus}, \texttt{662\_bus}, \texttt{685\_bus}, \texttt{1138\_bus}, \texttt{bcspwr05}, and \texttt{bcspwr06}. These small-to-medium matrices span least-squares and power-network problems and permit dense FP64 SVDs. For each matrix, we formed the rank-$k$ baseline and rounded the rank-$(k+1)$ factors ($U_{k+1}$, $\Sigma_{k+1}$, and $V_{k+1}^{\T}$) individually to FP32 or FP16 to obtain $\widehat A_{k+1}$. Reconstruction and error measurement used FP64. Unless otherwise noted, $E_k$, the new error, and $\eta$ are relative Frobenius norms. The default experiment used ranks $k\in\{2,5,10,20,40\}$; selected smaller matrices were also tested at larger ranks near the FP16 perturbation floor.

\begin{table}[!htbp]
\centering
\caption{Summary of matrix rank-compensation experiments.  A ``certified'' case is one in which \cref{eq:matrix_condition} holds.  The error ratio is $\|A-\widehat A_{k+1}\|_F/E_k$, so values below one mean that low precision with rank $k+1$ improves on the FP64 rank-$k$ baseline.}
\label{tab:matrix_results_summary}
\resizebox{\textwidth}{!}{%
\begin{tabular}{llrrrrr}
\toprule
Regime & Format & Cases & Certified & Wins & Avg. error ratio & Avg. storage ratio \\
\midrule
Default & FP32 & 50 & 50 & 50 & 0.963 & 0.588 \\
Default & FP16 & 50 & 50 & 50 & 0.963 & 0.294 \\
Tail stress & FP32 & 17 & 17 & 17 & 0.953 & 0.510 \\
Tail stress & FP16 & 17 & 10 & 15 & 0.976 & 0.255 \\
\bottomrule
\end{tabular}}
\end{table}

The default regime is truncation dominated: all FP32 and FP16 rank-$(k+1)$
approximations in \cref{tab:matrix_results_summary} strictly
improve on the FP64 rank-$k$ baseline, and all cases satisfy the non-increase
certificate \eqref{eq:matrix_condition}.

In the tail-rank tests, FP32 improves accuracy in all cases, but FP16 loses
accuracy for \texttt{494\_bus} at $k=300$ and $k=400$. In absolute terms, at
$k=300$, the FP64 baseline error is $E_{300} \approx 98.83$, and the FP16
perturbation for rank $301$ is $\eta \approx 21.81$. Since the remaining
singular tail does not decrease enough to offset $\eta$, the condition
\eqref{eq:matrix_condition} fails and the FP16 error increases to $100.15$.
Similarly at $k=400$, $E_{400} \approx 24.21$, and the FP16 error increases to
$32.25$. The sufficient condition fails in both cases.

At the largest tail-stress ranks, the SVD factors exceed dense storage. For
example, \texttt{494\_bus} has \(494^2=244{,}036\) dense entries, while a
rank-\(k\) SVD factorization stores \(k(494+494+1)=989k\) scalars. Thus the
SVD-factor representation exceeds dense storage for \(k\ge 247\). The storage
percentages reported for the tail-stress tests are relative to the FP64 rank-\(k\)
factor baseline, not to dense storage.

Condition \eqref{eq:matrix_condition} is sufficient but not necessary. Because it rests on the triangle inequality, it may miss some empirical accuracy improvements. For instance, \texttt{494\_bus} at $k=60 \to 61$ yields a ratio of $0.986$ but is not certified. The condition guarantees non-increase, while strict improvements may also occur beyond the certification boundary.

For the \texttt{ash85} matrix even at \(k=80\), FP16 rank \(81\)
remains certified. In the relative units used in the tail-stress experiment,
the baseline error \(6.73\times 10^{-3}\) decreases to \(4.46\times 10^{-3}\)
while the perturbation is \(\eta=3.65\times 10^{-4}\); the gain from the 81st
singular value still outweighs the FP16 perturbation.
The complete FP16 tail-stress grid is reported in
\cref{tab:tail_stress_detailed}.

\begin{table}[!htbp]
\centering
\caption{Detailed FP16 results for the tail-rank stress experiment. The error ratio is $\|A-\widehat A_{k+1}\|_F/E_k$, and values below 1 indicate that the low-precision model improves on the baseline. The matrix condition provides a sufficient certificate; non-certified cases may still be empirical wins.}
\label{tab:tail_stress_detailed}
\resizebox{\textwidth}{!}{%
\begin{tabular}{lrrrrrrr}
\toprule
Matrix & $k \to k+1$ & $E_k$ & $\eta$ & New Error & Ratio & Storage & Cert.? \\
\midrule
\texttt{ash85} & 20 $\to$ 21 & 4.27e-01 & 3.36e-04 & 4.13e-01 & 0.966 & 26.2\% & Yes \\
\texttt{ash85} & 40 $\to$ 41 & 2.34e-01 & 3.59e-04 & 2.27e-01 & 0.969 & 25.6\% & Yes \\
\texttt{ash85} & 60 $\to$ 61 & 1.01e-01 & 3.64e-04 & 9.54e-02 & 0.941 & 25.4\% & Yes \\
\texttt{ash85} & 80 $\to$ 81 & 6.73e-03 & 3.65e-04 & 4.46e-03 & 0.663 & 25.3\% & Yes \\
\texttt{ash219} & 20 $\to$ 21 & 7.62e-01 & 2.29e-04 & 7.52e-01 & 0.986 & 26.2\% & Yes \\
\texttt{ash219} & 40 $\to$ 41 & 5.68e-01 & 2.99e-04 & 5.58e-01 & 0.984 & 25.6\% & Yes \\
\texttt{ash219} & 60 $\to$ 61 & 3.76e-01 & 3.39e-04 & 3.66e-01 & 0.973 & 25.4\% & Yes \\
\texttt{ash219} & 80 $\to$ 81 & 1.28e-01 & 3.59e-04 & 1.12e-01 & 0.882 & 25.3\% & Yes \\
\texttt{494\_bus} & 20 $\to$ 21 & 6.59e-02 & 3.78e-04 & 6.29e-02 & 0.955 & 26.2\% & Yes \\
\texttt{494\_bus} & 40 $\to$ 41 & 3.78e-02 & 3.79e-04 & 3.71e-02 & 0.980 & 25.6\% & Yes \\
\texttt{494\_bus} & 60 $\to$ 61 & 2.58e-02 & 3.79e-04 & 2.55e-02 & 0.986 & 25.4\% & No \\
\texttt{494\_bus} & 80 $\to$ 81 & 1.95e-02 & 3.79e-04 & 1.93e-02 & 0.986 & 25.3\% & No \\
\texttt{494\_bus} & 100 $\to$ 101 & 1.51e-02 & 3.79e-04 & 1.49e-02 & 0.988 & 25.2\% & No \\
\texttt{494\_bus} & 150 $\to$ 151 & 8.36e-03 & 3.79e-04 & 8.28e-03 & 0.990 & 25.2\% & No \\
\texttt{494\_bus} & 200 $\to$ 201 & 5.06e-03 & 3.79e-04 & 5.03e-03 & 0.993 & 25.1\% & No \\
\texttt{494\_bus} & 300 $\to$ 301 & 1.72e-03 & 3.79e-04 & 1.74e-03 & 1.013 & 25.1\% & No \\
\texttt{494\_bus} & 400 $\to$ 401 & 4.21e-04 & 3.79e-04 & 5.61e-04 & 1.332 & 25.1\% & No \\
\bottomrule
\end{tabular}}
\end{table}

\FloatBarrier

\subsection{Downstream Matrix-Application Performance}
\label{sec:matrix-performance}

We next time matrix application with the corresponding resident factors under
the A100 protocol in \cref{sec:gpu-protocol}.
For a matrix approximation
$A_r=U_r\operatorname{diag}(s_r)V_r^{\T}$ and a batch
$X\in\mathbb{R}^{b\times n}$, the kernel evaluates
\begin{equation}
  Y=\bigl[(XV_r)\operatorname{diag}(s_r)\bigr]U_r^{\T},
  \label{eq:gpu_matrix_apply}
\end{equation}
using two preallocated matrix multiplications and an in-place column scaling.
The benchmark uses all ten SuiteSparse matrices in
\cref{subsec:matrix_experiments}, base rank 40, augmented rank 41, and
$b\in\{256,2048,8192\}$. A seeded standard-normal input is shared across
methods after every row is scaled to unit Euclidean norm, fixing the input range
instead of allowing arbitrary query scale to determine FP16 representability.

The FP64 rank-40 application is the baseline. Same-rank FP32 and FP16 controls
separate the effect of precision from the cost of rank augmentation, while the
rank-compensated FP32 and FP16 methods apply the rank-41 factors. Thus the
storage and approximation-error ratios in \cref{tab:matrix-speed} refer to the
same stored factors whose accuracy--memory trade-off was evaluated above.
Agreement between low-precision kernel output and FP64 evaluation of the stored
factors is assessed as described in \cref{sec:gpu-protocol}.

\begin{table}[!htbp]
  \centering
  \caption{NVIDIA A100 resident matrix-application timings at batch
  8192 for FP64 rank 40 and rank-compensated rank 41. Results are geometric
  means over the ten SuiteSparse matrices. Time is the
  geometric mean of per-task medians, and speedups use paired block medians;
  95\% intervals quantify variation across matrices. Storage and approximation
  error are ratios to the matching FP64 rank-40 representation, with error
  evaluated by FP64 contraction of the stored factors.}
  \label{tab:matrix-speed}
  \small
  \begin{tabular}{lrrrr}
\toprule
Method & Time (ms) & Speedup [95\% CI] & Storage & Error \\
\midrule
FP64 baseline & 0.102 & 1.00 [1.00, 1.00] & 1.0000 & 1.000 \\
FP32 same rank & 0.075 & 1.35 [1.31, 1.39] & 0.5000 & 1.000 \\
FP16 same rank & 0.044 & 2.32 [2.01, 2.67] & 0.2500 & 1.000 \\
FP32 rank-comp. & 0.080 & 1.28 [1.24, 1.31] & 0.5125 & 0.976 \\
FP16 rank-comp. & 0.048 & 2.12 [1.91, 2.31] & 0.2563 & 0.976 \\
\bottomrule
\end{tabular}

\end{table}

At batch 8192, the rank-compensated FP32 and FP16 methods attain
geometric-mean speedups of $1.28\times$ and $2.12\times$, respectively, while
reducing factor storage to $0.5125$ and $0.2563$ of the FP64 rank-40 baseline.
Their geometric-mean error ratio is $0.976$. The same-rank FP32 and FP16
controls reach $1.35\times$ and $2.32\times$, respectively. This comparison
isolates the modest runtime cost of adding one singular component.

Speedup depends strongly on batch size. At batch
256, the compensated FP32 and FP16 speedup factors are $0.99$ and $0.94$, so
neither method accelerates and fixed overhead dominates these small kernels.
At batch 2048, the factors rise to $1.04$ and $1.36$, and the largest batch
produces the gains reported above.

\begin{figure}[!htbp]
  \centering
  \includegraphics[width=0.76\linewidth]{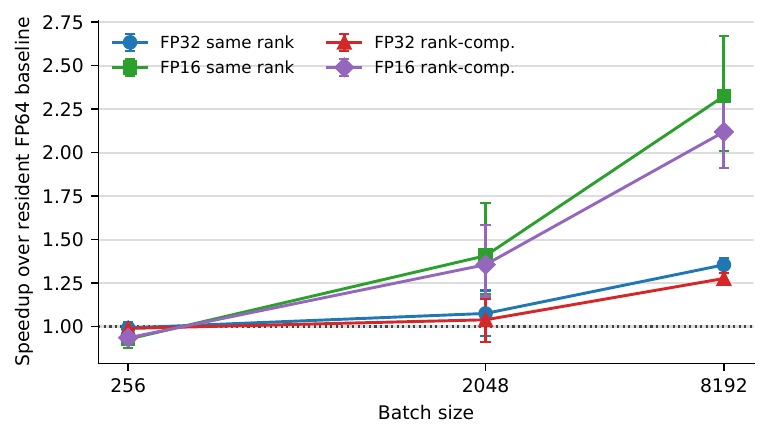}
  \caption{Geometric-mean NVIDIA A100 resident matrix-application speedup over
  the ten SuiteSparse matrices as batch size increases, using FP64 rank 40 and
  rank-compensated rank 41. Error bars are 95\% case-level
  bootstrap intervals. The rank-compensated methods become faster only when the
  workload is large enough to amortize fixed overhead.}
  \label{fig:matrix-speedup}
\end{figure}

At sufficiently large batches, the matrix accuracy--memory trade-off is retained
while downstream application becomes faster. These resident-kernel timings
exclude SVD construction, allocation, data transfer, and end-to-end latency.

\FloatBarrier

\section{Tensor-Train Preliminaries} \label{sec:tt_preliminaries}

We recall the TT definitions and stability properties used in the
rank-compensation analysis.

\subsection{\texorpdfstring{Tensor--Train}{Tensor-Train} (TT) Decomposition} \label{subsec:tt_decomposition}

A \(d\)-way tensor \(T \in \mathbb{R}^{n_1 \times n_2 \times \dots \times n_d}\) is a multi-dimensional array.
Its uncompressed storage requirement \(\prod_{\ell=1}^d n_\ell\) grows exponentially with the dimension~\(d\), a manifestation of the \emph{curse of dimensionality}.
For general background on tensor decompositions, see~\cite{KoldaBader2009}; for the TT format specifically, see~\cite{Oseledets2011}.
The Tensor--Train (TT) decomposition~\cite{Oseledets2011}, also known in quantum physics as the Matrix Product State (MPS)~\cite{Schollwock2011}, offers a concise representation for tensors that exhibit suitable low-rank structure.

In TT format the tensor is factorized into three-way \emph{cores} \(G_\ell \in \mathbb{R}^{R_{\ell-1}\times n_\ell \times R_\ell}\) for \(\ell = 1,\dots,d\).
Writing \(G_\ell(i_\ell) \in \mathbb{R}^{R_{\ell-1}\times R_\ell}\) for a slice obtained by fixing the physical index \(i_\ell\), entries of \(T\) are reconstructed via \(T(i_1,\dots,i_d)\;=\;G_1(i_1)\,G_2(i_2)\cdots G_d(i_d)\), with boundary ranks \(R_0=R_d=1\).
The storage cost is \(\sum_{\ell=1}^d R_{\ell-1} n_\ell R_\ell = O(d n R_{\max}^2)\) when \(n_\ell = n\) and \(R_{\max} = \max_\ell R_\ell\); thus, if the TT-ranks are moderate, the scaling is linear in \(d\).
The corresponding core-chain structure and its physical and rank indices are
shown in \cref{fig:tt_decomposition_diagram}.

\begin{figure}[htbp]
  \centering
  \begin{tikzpicture}[
      core/.style={draw, minimum width=1.05cm, minimum height=0.72cm, line width=1pt},
      edge/.style={line width=1pt},
      every node/.style={font=\small}
    ]
    \node[core] (g1) at (0,0) {$G_1$};
    \node[core] (g2) at (2.0,0) {$G_2$};
    \node (dots) at (3.65,0) {$\cdots$};
    \node[core] (gd) at (5.3,0) {$G_d$};
    \draw[edge] (-1.25,0) -- (g1.west);
    \node[anchor=south, yshift=2pt] at (-1.25,0) {$R_0=1$};
    \draw[edge] (g1.east) -- node[above] {$R_1$} (g2.west);
    \draw[edge] (g2.east) -- node[above] {$R_2$} (dots.west);
    \draw[edge] (dots.east) -- node[above] {$R_{d-1}$} (gd.west);
    \draw[edge] (gd.east) -- (6.55,0);
    \node[anchor=south, yshift=2pt] at (6.55,0) {$R_d=1$};
    \draw[edge] (g1.south) -- ++(0,-0.75) node[below] {$i_1$};
    \draw[edge] (g2.south) -- ++(0,-0.75) node[below] {$i_2$};
    \draw[edge] (gd.south) -- ++(0,-0.75) node[below] {$i_d$};
  \end{tikzpicture}
  \caption{Tensor--Train representation. Each box is a core
  $G_\ell\in\mathbb{R}^{R_{\ell-1}\times n_\ell\times R_\ell}$; vertical
  edges denote physical indices $i_\ell$, horizontal edges denote TT-rank
  indices $R_\ell$, and $R_0=R_d=1$.}
  \label{fig:tt_decomposition_diagram}
\end{figure}
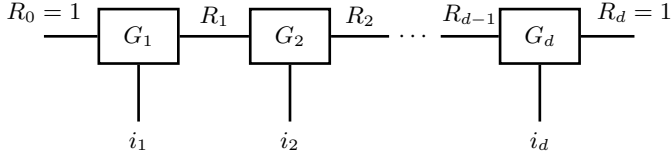

\paragraph{Frobenius norm}
We measure tensor errors in the Frobenius norm,
\begin{equation}
\|T\|_F =
\left(
  \sum_{i_1=1}^{n_1} \cdots \sum_{i_d=1}^{n_d}
  |T(i_1,\dots,i_d)|^2
\right)^{1/2}
\end{equation}

\paragraph{Low-rank approximation and error bound}
Let \(T_R\) be the TT approximation obtained by the \emph{TT-SVD} algorithm of Oseledets~\cite{Oseledets2011}, which sequentially applies truncated SVDs and retains \(R_\ell\) dominant singular values at each step. For the \(\ell\)th unfolding \(T^{\langle \ell\rangle}\), let \(\varepsilon_\ell\) be the Frobenius residual of its best rank-\(R_\ell\) matrix approximation. The accumulated unfolding-error bound for TT-SVD is~\cite[Theorem~2.2]{Oseledets2011}
\begin{equation}
E =\; \|T - T_R\|_F
    \;\le\;
    \Bigl(\,
      \sum_{\ell=1}^{d-1}
        \varepsilon_\ell^2
    \Bigr)^{1/2},
\label{eq:tt_svd_error}
\end{equation}
Because TT-SVD truncations are sequential, the unfolding errors are coupled.
The TT certificate therefore uses the realized gain \(S_{\mathrm{TT}}\), rather
than an exact analogue of the matrix singular-value identity.

\subsection{A basic building block}

\begin{lemma}[Componentwise forward error for a row--matrix product]\label{lem:dot}
Let \(x\in \mathbb{R}^{1\times m}\) and
\(A\in \mathbb{R}^{m\times n}\). Suppose each entry of
\(\widehat y=\fl(xA)\) is computed as a standard floating-point
dot product and \(m u<1\). Then
\[
\widehat y = xA + \Delta y,
\qquad
|\Delta y| \le \gamma_m |x|\,|A|,
\qquad
\gamma_m=\frac{m u}{1-m u},
\]
where the absolute values and inequality are componentwise.
\end{lemma}

\noindent This is the standard componentwise forward-error bound for dot
products~\cite{Higham2002}.

\subsection{Numerical stability of TT evaluations}

In TT contractions, the local dot-product perturbations are naturally controlled in an absolute/componentwise sense. Repeated application of Lemma~\ref{lem:dot} gives bounds involving products of absolute core slices rather than a simple relative error bound for the final tensor entry. In particular, a conservative perturbation scale has the form
\[
|\widehat T(i_1,\ldots,i_d)-T(i_1,\ldots,i_d)|
\lesssim
\gamma\, |G_1(i_1)|\,|G_2(i_2)|\cdots |G_d(i_d)|,
\]
where \(\gamma\) depends on the contraction order, the TT ranks, and the working
precision. This bound concerns low-precision contraction. The experiments
instead reconstruct rounded cores in FP64, so their certificates measure
storage quantization alone. Entrywise relative bounds require assumptions that
exclude severe cancellation.


\section{Rank-Augmented Tensor Trains}
\label{sec:theoretical_analysis}

The matrix theorem separates the argument into two quantities: a truncation
gain, measured by $\sigma_{k+1}^2$, and a low-precision perturbation, measured by
$\eta_{k+1}$. The TT analysis uses the same separation, but both quantities are
more delicate. Increasing all TT ranks by one does not simply restore one
independent singular component at each interface, because TT-SVD truncations are
sequential and coupled. Moreover, rounding TT cores is gauge
dependent~\cite{HoltzRohwedderSchneider2012,Schollwock2011}: two equivalent TT
representations of the same tensor can have different componentwise rounding
behavior.

Since TT-SVD is sequential and quasi-optimal rather than globally
optimal~\cite[Corollary~2.4]{Oseledets2011}, the theorem uses the realized gain
\(S_{\mathrm{TT}}\) and assumes that it is nonnegative.

\begin{theorem}[Conditional TT rank compensation]\label{thm:tt_conditional}
Let $T$ be a tensor and let $T_R$ be a TT approximation of multi-rank
$R=(R_1,\ldots,R_{d-1})$ with
\[
  E_R=\|T-T_R\|_F .
\]
Let \(R^+=(R_1^+,\ldots,R_{d-1}^+)\) be an augmented rank vector with
\(R_j^+\ge R_j\) for each \(j\), and let \(T_{R^+}\) be a rank-\(R^+\)
TT approximation of \(T\).  Define
\[
  E_{R^+}=\|T-T_{R^+}\|_F
\]
and the realized truncation gain
\begin{equation}\label{eq:tt_gain_assumption}
  S_{\mathrm{TT}} := E_R^2-E_{R^+}^2,
\end{equation}
and assume $S_{\mathrm{TT}}\ge 0$.  Let $\widehat T_{R^+}$ be a low-precision realization of $T_{R^+}$ and define
\[
  \eta_{\mathrm{TT}}=\|T_{R^+}-\widehat T_{R^+}\|_F .
\]
Then
\begin{equation}\label{eq:tt_universal}
  \|T-\widehat T_{R^+}\|_F
  \le
  \bigl(E_R^2-S_{\mathrm{TT}}\bigr)^{1/2}
  +\eta_{\mathrm{TT}}.
\end{equation}
In particular, if
\begin{equation}\label{eq:tt_condition}
  E_{R^+}+\eta_{\mathrm{TT}}\le E_R,
\end{equation}
then
\[
  \|T-\widehat T_{R^+}\|_F \le E_R .
\]
In exact arithmetic, \eqref{eq:tt_condition} is equivalent to the conjunction
of $\eta_{\mathrm{TT}}\le E_R$ and
\begin{equation}\label{eq:tt_condition_squared}
  S_{\mathrm{TT}}\ge 2\eta_{\mathrm{TT}}E_R-\eta_{\mathrm{TT}}^2.
\end{equation}
\end{theorem}

\begin{proof}
The proof is the same perturbation argument as in the matrix case.  By the triangle inequality and \eqref{eq:tt_gain_assumption},
\[
  \|T-\widehat T_{R^+}\|_F
  \le
  \|T-T_{R^+}\|_F+\|T_{R^+}-\widehat T_{R^+}\|_F
  \le
  \bigl(E_R^2-S_{\mathrm{TT}}\bigr)^{1/2}
  +\eta_{\mathrm{TT}}.
\]
Because $(E_R^2-S_{\mathrm{TT}})^{1/2}=E_{R^+}$,
\eqref{eq:tt_condition} makes this upper bound at most $E_R$. The equivalence
with \eqref{eq:tt_condition_squared} follows by the same rearrangement and
squaring argument as in the proof of \cref{thm:matrix_rank_compensation}.
\end{proof}

\noindent For numerical evaluation, we use \eqref{eq:tt_condition} directly,
avoiding the subtraction of nearly equal squared errors required to compute
\(S_{\mathrm{TT}}\) and evaluate \eqref{eq:tt_condition_squared}.

The TT result is a posteriori, unlike the matrix certificate. The matrix gain
$\sigma_{k+1}^2$ is known before rounding, whereas $S_{\mathrm{TT}}$ must be
measured after computing the augmented FP64 approximation. Thus
\cref{thm:tt_conditional} tests a computed representation but does not select
ranks in advance. An a priori lower bound for $S_{\mathrm{TT}}$ remains open.

For the 128 TT rank augmentations tested below,
\(S_{\mathrm{TT}}\ge0\) held to numerical tolerance. Each augmentation has FP32
and FP16 rounded-core evaluations, giving 256 precision-specific rows.

\subsection{Interpreting the TT Quantities}\label{subsec:discussion_error_bounds}

The practical comparison involves a truncation gain, a rounding perturbation,
and a storage constraint.

\paragraph{Truncation gain}\leavevmode\par
In the matrix case, the gain from rank $k$ to $k+1$ is exactly $\sigma_{k+1}^2$.  In TT format, a natural additive diagnostic for a general augmented vector is
\[
  S_{\mathrm{diag}}^{(R,R^+)}
  =
  \sum_{j=1}^{d-1}
  \sum_{\ell=1}^{R_j^+-R_j}
  \left(\sigma_{R_j+\ell}^{(j,+)}\right)^2,
\]
where $\sigma_q^{(j,+)}$ denotes the local singular values encountered at the
$j$th sequential unfolding in the augmented-rank TT-SVD run. The terms at
positions $R_j+1$ through $R_j^+$ lie beyond the base-rank cutoff in that
augmented run. They are not, in general, independently identifiable components
newly admitted relative to the base run, because the earlier sequential
truncations and hence later unfoldings can differ between the two runs. When
$R_j^+=R_j+\delta$ at every interface, we abbreviate this quantity as
$S_{\mathrm{diag}}^{(\delta)}$. For nominal uniform schedules with capped
ranks, the upper limit remains the actual increment \(R_j^+-R_j\). In the
controlled experiments below, $S_{\mathrm{diag}}^{(\delta)}$ does not
accurately track the realized gain $S_{\mathrm{TT}}$, suggesting that this
additive local diagnostic may be unreliable as a quantitative predictor.

\paragraph{Low-precision perturbation}\leavevmode\par
The perturbation $\eta_{\mathrm{TT}}$ should be measured after choosing and
reporting a controlled representation of the TT cores, such as a
left-orthogonal or right-orthogonal gauge. Such an orthogonal form is not a
unique canonical representation: residual orthogonal gauge transformations can
remain. It nevertheless controls arbitrary inter-core scaling. This is
important because inserting $MM^{-1}$ between adjacent cores leaves the
represented tensor unchanged but may change the effect of rounding individual
entries. Controlled orthogonal gauges are standard in TT/MPS
algorithms~\cite{Oseledets2011,HoltzRohwedderSchneider2012,Schollwock2011}. A
practical TT algorithm should therefore round only after bringing the
representation into a reported controlled gauge. In the experiments,
$\widehat T_{R^+}$ is the tensor represented by rounded stored cores but
reconstructed in FP64, so $\eta_{\mathrm{TT}}$ measures storage quantization
and excludes low-precision contraction arithmetic. Orthogonalizing a TT with
$d$ cores and maximum rank $R$ requires $O(d n R^3)$ operations. This overhead
is standard in TT algebra, but it should be accounted for in applications that
seek runtime gains in addition to storage reduction.

\paragraph{Storage trade-off}\leavevmode\par
For comparable mode sizes $n_j\approx n$ and ranks $R_j\approx R$, TT storage is proportional to $b\,d\,nR^2$, where $b$ is the number of bits per scalar.  Moving from FP64 rank $R$ to lower precision rank $R+\delta$ gives the approximate ratio
\begin{equation}\label{eq:tt_storage_ratio}
  \frac{\mathrm{storage}_{\mathrm{low}}(R+\delta)}
       {\mathrm{storage}_{\mathrm{FP64}}(R)}
  \approx
  \frac{b_{\mathrm{low}}}{64}
  \left(1+\frac{\delta}{R}\right)^2 .
\end{equation}
For example, $R=4\to8$ in FP32 has a ratio greater than one. Because
\eqref{eq:tt_storage_ratio} omits boundary-core effects, it is an asymptotic
guide; for a uniform three-core TT the exact ratio is
\[
  \frac{b_{\mathrm{low}}}{64}
  \frac{(R+\delta)(R+\delta+2)}{R(R+2)}.
\]
All reported ratios use the exact core dimensions. Large-$\delta$ cases probe
the certification boundary and are practical only when the exact ratio is below
one.

\section{Tensor-Train Experiments} \label{sec:tt_experiments}

\subsection{A Posteriori Diagnostic Algorithm}
\label{subsec:tt-diagnostic}

\Cref{alg:tt_compensation} applies the certificate to a uniform rank increment;
capped schedules use the general vector \(R^+\succeq R\). The algorithm is an
a posteriori acceptance test for a computed rank-augmented representation.

\begin{algorithm}[htbp]
\caption{A posteriori rank-compensated TT check}
\label{alg:tt_compensation}
\begin{algorithmic}[1]
\REQUIRE FP64 tensor $T$, target base ranks $R$, increment $\delta$, precision $b_{\mathrm{low}}$
\STATE Compute baseline FP64 TT approximation $T_R$ via TT-SVD.
\STATE Compute augmented FP64 TT approximation $T_{R+\delta}$ via TT-SVD.
\STATE Left-orthogonalize the cores of $T_{R+\delta}$ to stabilize the gauge.
\STATE Round the orthogonalized cores to $b_{\mathrm{low}}$ to form $\widehat T_{R+\delta}$.
\STATE Measure $E_R = \|T - T_R\|_F$, $E_{R+\delta}=\|T-T_{R+\delta}\|_F$, and $S_{\mathrm{TT}} = E_R^2-E_{R+\delta}^2$.
\STATE Measure $\eta_{\mathrm{TT}} = \|T_{R+\delta} - \widehat T_{R+\delta}\|_F$.
\IF{$E_{R+\delta}+\eta_{\mathrm{TT}} \le E_R$ and storage ratio $< 1$}
    \RETURN $\widehat T_{R+\delta}$ (Certified non-increase and memory win)
\ELSIF{$E_{R+\delta}+\eta_{\mathrm{TT}} \le E_R$}
    \RETURN $\widehat T_{R+\delta}$ (Certified accuracy non-increase only)
\ELSE
    \RETURN $T_R$ (Fallback to FP64 baseline)
\ENDIF
\end{algorithmic}
\end{algorithm}

\subsection{Three-Way Synthetic Experiments}
\label{subsec:tt-three-way}

\begin{sloppypar}
We first consider two dense three-way tensors with different truncation
behavior. The \texttt{hilbert\_3d} tensor has size $100\times100\times100$ and
entries $T(i,j,k)=(i+j+k-2)^{-1}$. The \texttt{decay\_3d} tensor has size
$60\times60\times60$ and entries
$T(i,j,k)=(1+|i-j|+|j-k|+|i-k|)^{-1}$.
\end{sloppypar}

We applied standard TT-SVD to each dense tensor and reconstructed the resulting TT tensors in FP64 for error measurement across a schedule of base ranks ($R \in \{4, 8, 16\}$). We then augmented the ranks uniformly by increments ($\delta \in \{1, 4\}$). To control $\eta_{\mathrm{TT}}$, the augmented cores were explicitly left-orthogonalized before being rounded to FP32 and FP16. Memory costs were calculated exactly using core dimensions, and errors were measured in the relative Frobenius norm. The realized truncation gain $S_{\mathrm{TT}}$ was obtained from the measured FP64 reconstruction errors, and $S_{\mathrm{diag}}^{(\delta)}$ was computed from positions $R_j+1$ through $R_j+\delta$ of the local spectra encountered in the augmented TT-SVD run.

Among the 12 left-gauge trials per precision in \cref{tab:tt_summary}, FP32 and
FP16 yield 6 and 8 practical wins, respectively. Failures occur when the
rounding perturbation is comparable to the truncation error.

Across the 24 left-gauge trials, no certified row exceeds the baseline error
beyond numerical tolerance. Certification concerns accuracy only; a practical win also requires
a strict error reduction and a storage ratio below one. For
\texttt{hilbert\_3d}, the rank-8 and rank-16 baselines are already extremely
accurate ($E_{16}\approx2.25\times10^{-12}$), while the FP16 perturbation is
about $3.44\times10^{-4}$. The condition fails and the rounded model loses.
The same behavior occurs for FP32 once the FP64 error falls below the FP32
perturbation scale.

\begin{table}[!htbp]
\centering
\caption{Three-way synthetic TT results in the left-orthogonal gauge. An accuracy-and-memory win strictly reduces both error and storage relative to the baseline.}
\label{tab:tt_summary}
\resizebox{\textwidth}{!}{%
\begin{tabular}{llrrrrrr}
\toprule
Dataset & Precision & Trials & Acc. Wins & Mem. Wins & Acc.+Mem. Wins & Certified & Certified, no strict win \\
\midrule
\texttt{hilbert\_3d} & FP32 & 6 & 4 & 4 & 2 & 4 & 0 \\
\texttt{hilbert\_3d} & FP16 & 6 & 2 & 6 & 2 & 2 & 0 \\
\texttt{decay\_3d} & FP32 & 6 & 6 & 4 & 4 & 6 & 0 \\
\texttt{decay\_3d} & FP16 & 6 & 6 & 6 & 6 & 6 & 0 \\
\bottomrule
\end{tabular}}
\end{table}

The individual three-way configurations are reported in
\cref{tab:tt_three_way_detail}.
\begin{table}[!htbp]
\centering
\caption{Detailed three-way synthetic TT rank-compensation results. Baseline is FP64 rank-$R$. The augmented rank-$R^+$ model is rounded to the listed precision. $\eta$ is the low-precision perturbation. Metrics are scaled by $\|T\|_F$.}
\label{tab:tt_three_way_detail}
\resizebox{\textwidth}{!}{%
\begin{tabular}{lrrlrrrrrl}
\toprule
Dataset & $R$ & $R^+$ & Prec. & $E_R$ (Base) & New Error & $\eta$ & Err. Ratio & Mem. Ratio & Cert.? \\
\midrule
decay\_3d & 4 & 5 & FP16 & 5.52e-01 & 5.17e-01 & 3.20e-04 & \textbf{0.94} & \textbf{0.36} & Yes \\
decay\_3d & 4 & 8 & FP16 & 5.52e-01 & 4.45e-01 & 3.37e-04 & \textbf{0.81} & \textbf{0.83} & Yes \\
decay\_3d & 8 & 9 & FP16 & 4.45e-01 & 4.27e-01 & 3.28e-04 & \textbf{0.96} & \textbf{0.31} & Yes \\
decay\_3d & 8 & 12 & FP16 & 4.45e-01 & 3.85e-01 & 3.40e-04 & \textbf{0.86} & \textbf{0.53} & Yes \\
decay\_3d & 16 & 17 & FP16 & 3.43e-01 & 3.35e-01 & 3.51e-04 & \textbf{0.97} & \textbf{0.28} & Yes \\
decay\_3d & 16 & 20 & FP16 & 3.43e-01 & 3.11e-01 & 3.66e-04 & \textbf{0.91} & \textbf{0.38} & Yes \\
hilbert\_3d & 4 & 5 & FP16 & 6.15e-03 & 1.39e-03 & 3.45e-04 & \textbf{0.23} & \textbf{0.36} & Yes \\
hilbert\_3d & 4 & 8 & FP16 & 6.15e-03 & 3.45e-04 & 3.44e-04 & \textbf{0.06} & \textbf{0.83} & Yes \\
hilbert\_3d & 8 & 9 & FP16 & 9.38e-06 & 3.44e-04 & 3.44e-04 & 36.73 & \textbf{0.31} & No \\
hilbert\_3d & 8 & 12 & FP16 & 9.38e-06 & 3.44e-04 & 3.44e-04 & 36.73 & \textbf{0.53} & No \\
hilbert\_3d & 16 & 17 & FP16 & 2.25e-12 & 3.44e-04 & 3.44e-04 & $>1000$ & \textbf{0.28} & No \\
hilbert\_3d & 16 & 20 & FP16 & 2.25e-12 & 3.44e-04 & 3.44e-04 & $>1000$ & \textbf{0.38} & No \\
decay\_3d & 4 & 5 & FP32 & 5.52e-01 & 5.17e-01 & 3.97e-08 & \textbf{0.94} & \textbf{0.73} & Yes \\
decay\_3d & 4 & 8 & FP32 & 5.52e-01 & 4.45e-01 & 4.24e-08 & \textbf{0.81} & 1.67 & Yes \\
decay\_3d & 8 & 9 & FP32 & 4.45e-01 & 4.27e-01 & 4.05e-08 & \textbf{0.96} & \textbf{0.62} & Yes \\
decay\_3d & 8 & 12 & FP32 & 4.45e-01 & 3.85e-01 & 4.31e-08 & \textbf{0.86} & 1.05 & Yes \\
decay\_3d & 16 & 17 & FP32 & 3.43e-01 & 3.35e-01 & 4.21e-08 & \textbf{0.97} & \textbf{0.56} & Yes \\
decay\_3d & 16 & 20 & FP32 & 3.43e-01 & 3.11e-01 & 4.23e-08 & \textbf{0.91} & \textbf{0.76} & Yes \\
hilbert\_3d & 4 & 5 & FP32 & 6.15e-03 & 1.34e-03 & 4.36e-08 & \textbf{0.22} & \textbf{0.73} & Yes \\
hilbert\_3d & 4 & 8 & FP32 & 6.15e-03 & 9.38e-06 & 4.10e-08 & \textbf{0.00} & 1.67 & Yes \\
hilbert\_3d & 8 & 9 & FP32 & 9.38e-06 & 1.61e-06 & 4.10e-08 & \textbf{0.17} & \textbf{0.62} & Yes \\
hilbert\_3d & 8 & 12 & FP32 & 9.38e-06 & 4.15e-08 & 4.10e-08 & \textbf{0.00} & 1.05 & Yes \\
hilbert\_3d & 16 & 17 & FP32 & 2.25e-12 & 4.10e-08 & 4.10e-08 & $>1000$ & \textbf{0.56} & No \\
hilbert\_3d & 16 & 20 & FP32 & 2.25e-12 & 4.10e-08 & 4.10e-08 & $>1000$ & \textbf{0.76} & No \\
\bottomrule
\end{tabular}}
\end{table}

\FloatBarrier

The additive diagnostic does not reproduce the realized TT truncation gain in
these examples. $S_{\mathrm{diag}}^{(\delta)}$ overestimated the realized gain
$S_{\mathrm{TT}}$. Over the left-gauge trials with $S_{\mathrm{TT}}>0$, the
mean and median values of
$S_{\mathrm{diag}}^{(\delta)}/S_{\mathrm{TT}}$ were 1.218 and 1.211,
respectively, suggesting that local singular-value diagnostics are not reliable
additive predictors of the global TT rank-augmentation gain. Because TT-SVD
already returns left-orthogonal cores, the additional orthogonalization did not
change these perturbations.

\subsection{Six-Way Synthetic Experiments}
\label{subsec:tt-six-way}

We applied the same procedure to two six-way tensors of shape \(12^6\):
\texttt{hilbert\_6d}, with entries
\(1/(i_1+\cdots+i_6-5)\), and \texttt{decay\_6d}, with entries
\((1+\sum_{p<q}|i_p-i_q|)^{-1}\). The rank schedule used
\(R\in\{4,8\}\) and \(\delta\in\{1,4\}\). The results in
\Cref{tab:synthetic_tt_d6_summary} show 4 out of
8 FP32 practical wins and 6 out of 8 FP16 practical wins; no certified row has
error above the baseline beyond tolerance. The two FP16 losses occur for \texttt{hilbert\_6d}
at \(R=8\), where the FP64 baseline error is below the FP16 rounded-core
perturbation. For these six-way trials, the mean and median values of
\(S_{\mathrm{diag}}^{(\delta)}/S_{\mathrm{TT}}\) are 1.890 and 1.752,
respectively, compared with 1.218 and 1.211 for the three-way synthetic
trials. The additive diagnostic therefore agrees less closely with the realized
gain for these six-way cases. However, the three-way and six-way test sets also
differ in tensor definitions, mode sizes, and rank schedules. This comparison
does not isolate tensor order and cannot attribute the larger ratios to order
alone.

\begin{table}[!htbp]
\centering
\caption{Six-way controlled synthetic TT rank-compensation summary.}
\label{tab:synthetic_tt_d6_summary}
\resizebox{\textwidth}{!}{%
\begin{tabular}{llrrrrrr}
\toprule
Dataset & Precision & Trials & Acc. Wins & Mem. Wins & Acc.+Mem. Wins & Certified & Certified, no strict win \\
\midrule
hilbert\_6d & FP32 & 4 & 4 & 2 & 2 & 4 & 0 \\
hilbert\_6d & FP16 & 4 & 2 & 4 & 2 & 2 & 0 \\
decay\_6d & FP32 & 4 & 4 & 2 & 2 & 4 & 0 \\
decay\_6d & FP16 & 4 & 4 & 4 & 4 & 4 & 0 \\
\bottomrule
\end{tabular}}
\end{table}

The corresponding row-level results are given in
\cref{tab:synthetic_tt_d6_detail}.
\begin{table}[!htbp]
\centering
\caption{Six-way synthetic TT rank-compensation results. Baseline is FP64 rank-$R$. The augmented rank-$R^+$ model is rounded to the listed precision. Metrics are scaled by $\|T\|_F$.}
\label{tab:synthetic_tt_d6_detail}
\resizebox{\textwidth}{!}{%
\begin{tabular}{lrrlrrrrrl}
\toprule
Dataset & $R$ & $R^+$ & Prec. & $E_R$ & New Error & $\eta$ & Err. Ratio & Mem. Ratio & Cert. \\
\midrule
hilbert\_6d & 4 & 5 & FP32 & 7.03e-04 & 9.39e-05 & 6.08e-08 & 0.13 & 0.76 & Yes \\
hilbert\_6d & 4 & 5 & FP16 & 7.03e-04 & 5.52e-04 & 5.44e-04 & 0.78 & 0.38 & Yes \\
hilbert\_6d & 4 & 8 & FP32 & 7.03e-04 & 1.14e-07 & 6.01e-08 & 0.00 & 1.89 & Yes \\
hilbert\_6d & 4 & 8 & FP16 & 7.03e-04 & 5.44e-04 & 5.44e-04 & 0.77 & 0.94 & Yes \\
hilbert\_6d & 8 & 9 & FP32 & 9.73e-08 & 6.06e-08 & 6.01e-08 & 0.62 & 0.63 & Yes \\
hilbert\_6d & 8 & 9 & FP16 & 9.73e-08 & 5.44e-04 & 5.44e-04 & 5585.48 & 0.31 & No \\
hilbert\_6d & 8 & 12 & FP32 & 9.73e-08 & 6.01e-08 & 6.01e-08 & 0.62 & 1.10 & Yes \\
hilbert\_6d & 8 & 12 & FP16 & 9.73e-08 & 5.44e-04 & 5.44e-04 & 5585.48 & 0.55 & No \\
decay\_6d & 4 & 5 & FP32 & 1.31e-01 & 1.16e-01 & 5.48e-08 & 0.89 & 0.76 & Yes \\
decay\_6d & 4 & 5 & FP16 & 1.31e-01 & 1.16e-01 & 5.24e-04 & 0.89 & 0.38 & Yes \\
decay\_6d & 4 & 8 & FP32 & 1.31e-01 & 1.01e-01 & 7.37e-08 & 0.77 & 1.89 & Yes \\
decay\_6d & 4 & 8 & FP16 & 1.31e-01 & 1.01e-01 & 5.88e-04 & 0.77 & 0.94 & Yes \\
decay\_6d & 8 & 9 & FP32 & 1.01e-01 & 9.62e-02 & 7.92e-08 & 0.95 & 0.63 & Yes \\
decay\_6d & 8 & 9 & FP16 & 1.01e-01 & 9.62e-02 & 6.36e-04 & 0.95 & 0.31 & Yes \\
decay\_6d & 8 & 12 & FP32 & 1.01e-01 & 7.94e-02 & 7.22e-08 & 0.79 & 1.10 & Yes \\
decay\_6d & 8 & 12 & FP16 & 1.01e-01 & 7.94e-02 & 4.94e-04 & 0.79 & 0.55 & Yes \\
\bottomrule
\end{tabular}}
\end{table}

\FloatBarrier

\subsection{Public Hyperspectral and FROSTT Benchmarks}
\label{subsec:tt-public}

The dense benchmarks are Indian Pines corrected
\((145\times145\times200)\), Salinas-A corrected
\((83\times86\times204)\), Salinas corrected
\((512\times217\times204)\), and Pavia University
\((610\times340\times103)\) after exclusion of no-information samples, all from
the GIC repository~\cite{GICHyperspectralScenes}. The sparse benchmarks are
dense subtensors of the FROSTT \texttt{uber-pickups} and
\texttt{chicago-crime} tensors~\cite{FROSTT}.

For each FROSTT tensor, we selected the most active indices in each mode by
marginal nonzero count and formed the induced dense subtensor. Every public tensor was converted to FP64
and normalized by its Frobenius norm before TT-SVD, so the reported errors are
relative Frobenius errors.

The public experiments used the same measured quantities as the controlled tests. For the smaller hyperspectral cubes we used \(R\in\{2,4,8,16,32\}\) and \(\delta\in\{1,2,4\}\). For the larger hyperspectral cubes and FROSTT subtensors we used \(R\in\{4,8,16\}\) and \(\delta\in\{1,4\}\). For the Salinas-A tail-rank stress test, we used nominal ranks \(R_{\rm nom}\in\{48,64,80,96,128,160\}\) with \(\delta=1\), capping each TT rank by the corresponding valid unfolding dimension. Thus, at high nominal ranks the actual rank vector is not uniform; for example, \(R_{\rm nom}=160\) gives \(R=(83,160)\) and \(R^+=(83,161)\).

Across the 60 public trials per precision in
\cref{tab:public_tensor_summary}, FP32 improves accuracy in every case and gives
44 practical wins; FP16 improves accuracy in 56 cases and gives 54 practical wins. Its four
losses occur in the Salinas-A tail-rank test, where \(E_R\) falls to
\(6.23\times10^{-4}\)--\(1.54\times10^{-3}\) while
\(\eta_{\mathrm{TT}}\approx3.20\times10^{-4}\). One uncertified FP16 case still
improves accuracy, consistent with the condition being sufficient but not necessary.

In \cref{fig:public_eta_er}, the four failed FP16 Salinas-A cases occur where
the measured perturbation is no longer negligible relative to the FP64
truncation error.

\begin{figure}[!htbp]
  \centering
  \includegraphics[width=0.68\linewidth]{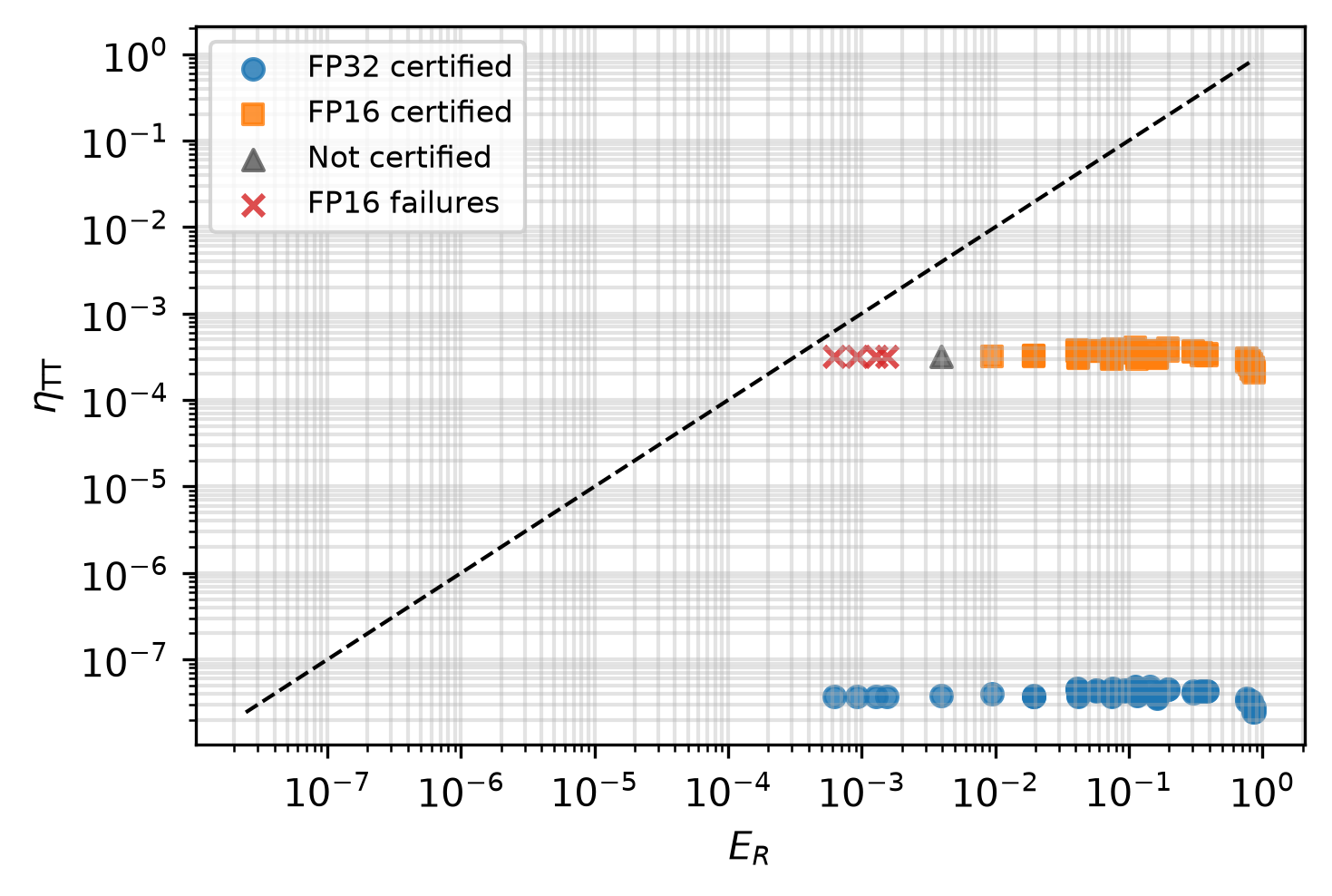}
  \caption{Public tensor certification behavior. Each point compares the
  storage-quantization perturbation \(\eta_{\mathrm{TT}}\) with the FP64
  baseline error \(E_R\); the dashed line \(\eta_{\mathrm{TT}}=E_R\) is a
  necessary condition for \eqref{eq:tt_condition_squared}. The categories
  contain 60 certified FP32 rows, 55 certified FP16 rows, one uncertified FP16
  accuracy win, and four uncertified FP16 losses.}
  \label{fig:public_eta_er}
\end{figure}

For example, the Salinas-A \(R_{\rm nom}=2\) to \(R^+=(6,6)\) FP16 row is
certified and reduces error, but its memory ratio is \(1.31\). The generalized
diagnostic \(S_{\mathrm{diag}}^{(R,R^+)}\) remains close to
\(S_{\mathrm{TT}}\) for the dense hyperspectral cubes but overestimates the
realized gain on the FROSTT subtensors, with mean ratios about \(2.45\) for
\texttt{uber-pickups} and \(3.06\) for \texttt{chicago-crime}.

\begin{table}[!htbp]
\centering
\caption{Public tensor benchmark summary. Trials are reported per precision, and P.W. denotes practical storage-and-accuracy wins.}
\label{tab:public_tensor_summary}
\small
\begin{tabular}{lcrrr}
\toprule
Dataset & Trials & $R_{\rm nom}$ & FP32 P.W. & FP16 P.W. \\
\midrule
FROSTT crime & 6 & {4,8,16} & 4 & 6 \\
FROSTT uber & 6 & {4,8,16} & 4 & 6 \\
Indian Pines & 15 & {2,4,8,16,32} & 11 & 14 \\
PaviaU & 6 & {4,8,16} & 4 & 6 \\
Salinas-A & 15 & {2,4,8,16,32} & 11 & 14 \\
Salinas-A tail & 6 & {48,64,80,96,128,160} & 6 & 2 \\
Salinas & 6 & {4,8,16} & 4 & 6 \\
\bottomrule
\end{tabular}
\end{table}

Selected errors and outcomes, the certification confusion matrix, and all four
FP16 failures are reported in
\cref{tab:public_tensor_detail_errors,tab:public_tensor_detail_outcomes,tab:public_tensor_confusion,tab:public_tensor_failures}.
\begin{table}[!htbp]
\centering
\caption{Selected public tensor rank-compensation errors. The baseline is the FP64 representation at the actual rank $R$.}
\label{tab:public_tensor_detail_errors}
\small
\setlength{\tabcolsep}{3pt}
\begin{tabular}{lllrrrr}
\toprule
Dataset & $R\to R^+$ & Prec. & $E_R$ & New error & $\eta$ & Err. ratio \\
\midrule
Indian Pines & $(2,2)$\(\to\)$(3,3)$ & FP16 & 1.11e-01 & 9.96e-02 & 3.69e-04 & 0.90 \\
Salinas-A & $(2,2)$\(\to\)$(6,6)$ & FP16 & 1.63e-01 & 8.96e-02 & 3.14e-04 & 0.55 \\
Salinas & $(4,4)$\(\to\)$(8,8)$ & FP16 & 1.97e-01 & 1.44e-01 & 3.97e-04 & 0.73 \\
PaviaU & $(4,4)$\(\to\)$(8,8)$ & FP16 & 3.84e-01 & 3.50e-01 & 3.47e-04 & 0.91 \\
FROSTT uber & $(4,4,4)$\(\to\)$(8,8,8)$ & FP16 & 8.27e-01 & 7.96e-01 & 2.64e-04 & 0.96 \\
FROSTT crime & $(4,4,4)$\(\to\)$(8,8,8)$ & FP16 & 8.70e-01 & 8.50e-01 & 2.17e-04 & 0.98 \\
Salinas-A tail & $(64,64)$\(\to\)$(65,65)$ & FP16 & 3.96e-03 & 3.74e-03 & 3.20e-04 & 0.94 \\
Salinas-A tail & $(80,80)$\(\to\)$(81,81)$ & FP16 & 1.54e-03 & 1.56e-03 & 3.20e-04 & 1.01 \\
Salinas-A tail & $(83,160)$\(\to\)$(83,161)$ & FP16 & 6.23e-04 & 6.91e-04 & 3.20e-04 & 1.11 \\
Salinas-A tail & $(83,160)$\(\to\)$(83,161)$ & FP32 & 6.23e-04 & 6.13e-04 & 3.73e-08 & 0.98 \\
\bottomrule
\end{tabular}
\end{table}

\begin{table}[!htbp]
\centering
\caption{Storage and certification outcomes for the selected public tensor cases in \cref{tab:public_tensor_detail_errors}.}
\label{tab:public_tensor_detail_outcomes}
\small
\begin{tabular}{llrrrr}
\toprule
Dataset & Prec. & $R_{\rm nom}$ & Err. ratio & Mem. ratio & Cert./Prac. \\
\midrule
Indian Pines & FP16 & 2 & 0.90 & 0.46 & Yes/Yes \\
Salinas-A & FP16 & 2 & 0.55 & 1.31 & Yes/No \\
Salinas & FP16 & 4 & 0.73 & 0.77 & Yes/Yes \\
PaviaU & FP16 & 4 & 0.91 & 0.83 & Yes/Yes \\
FROSTT uber & FP16 & 4 & 0.96 & 0.87 & Yes/Yes \\
FROSTT crime & FP16 & 4 & 0.98 & 0.89 & Yes/Yes \\
Salinas-A tail & FP16 & 64 & 0.94 & 0.26 & No/Yes \\
Salinas-A tail & FP16 & 80 & 1.01 & 0.26 & No/No \\
Salinas-A tail & FP16 & 160 & 1.11 & 0.25 & No/No \\
Salinas-A tail & FP32 & 160 & 0.98 & 0.50 & Yes/Yes \\
\bottomrule
\end{tabular}
\end{table}

\begin{table}[!htbp]
\centering
\caption{Public tensor certification outcomes. A win is a strict accuracy improvement over the FP64 baseline; certification refers to the sufficient non-increase condition.}
\label{tab:public_tensor_confusion}
\resizebox{0.9\textwidth}{!}{%
\begin{tabular}{lrrrrr}
\toprule
Precision & Trials & Cert.+Win & Cert.+Non-win & Not Cert.+Win & Not Cert.+Non-win \\
\midrule
FP16 & 60 & 55 & 0 & 1 & 4 \\
FP32 & 60 & 60 & 0 & 0 & 0 \\
\bottomrule
\end{tabular}}
\end{table}

\begin{table}[!htbp]
\centering
\caption{The four FP16 public tensor failure cases. In each case, the measured truncation gain is below the certification threshold.}
\label{tab:public_tensor_failures}
\small
\setlength{\tabcolsep}{3pt}
\begin{tabular}{llllrrrr}
\toprule
Dataset & $R_{\rm nom}$ & $R$ actual & $R^+$ actual & $E_R$ & $\eta$ & New error & Err. ratio \\
\midrule
Salinas-A tail & 160 & $(83,160)$ & $(83,161)$ & 6.23e-04 & 3.20e-04 & 6.91e-04 & 1.11 \\
Salinas-A tail & 128 & $(83,128)$ & $(83,129)$ & 9.22e-04 & 3.20e-04 & 9.67e-04 & 1.05 \\
Salinas-A tail & 96 & $(83,96)$ & $(83,97)$ & 1.28e-03 & 3.20e-04 & 1.30e-03 & 1.02 \\
Salinas-A tail & 80 & $(80,80)$ & $(81,81)$ & 1.54e-03 & 3.20e-04 & 1.56e-03 & 1.01 \\
\bottomrule
\end{tabular}
\end{table}

\FloatBarrier
\newpage

\subsection{Downstream Tensor-Train Reconstruction Performance}
\label{sec:tt-performance}

We measure dense reconstruction from the stored cores under the A100 protocol
in \cref{sec:gpu-protocol}. If
$G^{(\ell)}\in\mathbb{R}^{r_{\ell-1}\times n_\ell\times r_\ell}$,
the kernel reshapes the first core as
$Z_1\in\mathbb{R}^{n_1\times r_1}$ and each later core as
$K_\ell\in\mathbb{R}^{r_{\ell-1}\times(n_\ell r_\ell)}$. It then performs the
sequence
\begin{equation}
  B_\ell=Z_{\ell-1}K_\ell,
  \qquad
  Z_\ell=\operatorname{reshape}\!\left(
    B_\ell,\left(\prod_{j=1}^{\ell}n_j\right)\times r_\ell
  \right),
  \quad \ell=2,\ldots,d,
  \label{eq:gpu_tt_reconstruction}
\end{equation}
using preallocated output buffers. Since $r_d=1$, the final buffer is the
vectorized dense tensor.

The public timing grid contains the four complete hyperspectral cubes (Indian
Pines corrected, Salinas-A corrected, Salinas corrected, and Pavia University)
with base and augmented rank schedules $32\to36$, and the two top-active
FROSTT subtensors (Uber pickups and Chicago crime) with schedules $16\to20$.
Two six-way synthetic stress tests with schedules $8\to12$ are excluded from
the six-public-tensor aggregate.
For each public tensor, the FP64 base-rank reconstruction is compared with
same-rank FP32 and FP16 controls and rank-compensated FP32 and FP16
reconstructions. Storage and error are evaluated for the same cores used in the
timings.

\begin{table}[!htbp]
  \centering
  \caption{NVIDIA A100 resident TT-reconstruction timings for the six
  public tensors, using ranks $32\to36$ for the four complete hyperspectral
  cubes and $16\to20$ for the two FROSTT subtensors. Time is the geometric mean of per-task
  medians, and speedups use paired block medians; 95\% intervals quantify
  variation across tensors. Storage and approximation error are ratios to the
  matching FP64 base-rank representation, with error evaluated by FP64
  contraction of the stored cores.}
  \label{tab:tt-speed}
  \small
  \begin{tabular}{lrrrr}
\toprule
Method & Time (ms) & Speedup [95\% CI] & Storage & Error \\
\midrule
FP64 baseline & 0.085 & 1.00 [1.00, 1.00] & 1.0000 & 1.000 \\
FP32 same rank & 0.053 & 1.61 [1.38, 1.83] & 0.5000 & 1.000 \\
FP16 same rank & 0.044 & 1.93 [1.40, 2.64] & 0.2500 & 1.000 \\
FP32 rank-comp. & 0.062 & 1.38 [1.26, 1.49] & 0.6715 & 0.938 \\
FP16 rank-comp. & 0.044 & 1.94 [1.41, 2.61] & 0.3358 & 0.938 \\
\bottomrule
\end{tabular}

\end{table}

The rank-compensated FP32 and FP16 reconstructions in
\cref{tab:tt-speed,fig:tt-reconstruction-speedup} attain geometric-mean speedups of
$1.38\times$ and $1.94\times$, respectively. Their geometric-mean storage
ratios are $0.672$ and $0.336$, and their geometric-mean approximation-error
ratio is $0.938$. The same-rank FP32 and FP16 controls attain $1.61\times$ and
$1.93\times$, respectively. At the dataset level, rank-compensated FP16
speedups range from $1.04\times$ on Salinas-A to $3.12\times$ on Salinas,
while the FP32 range is $1.14\times$ to $1.54\times$.

\begin{figure}[!htbp]
  \centering
  \includegraphics[width=0.82\linewidth]{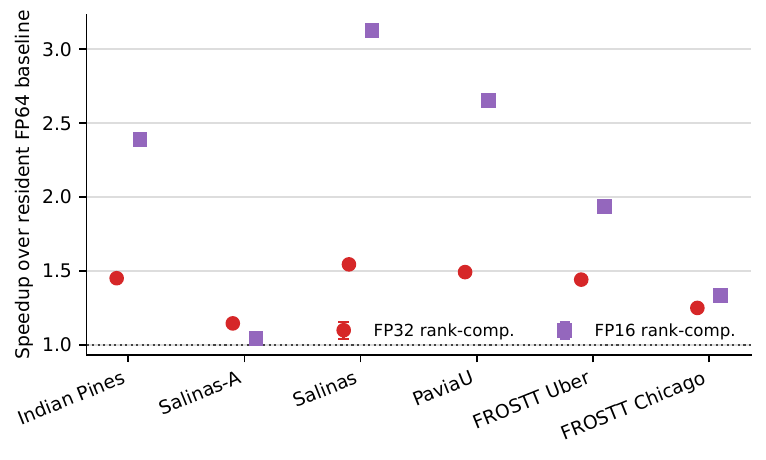}
  \caption{Per-dataset NVIDIA A100 resident TT-reconstruction speedup for the
  rank-compensated methods on the six public tensors, using ranks $32\to36$ for
  the hyperspectral cubes and $16\to20$ for the FROSTT subtensors. Error bars are 95\%
  bootstrap intervals over the 15 paired block ratios; several are smaller
  than the markers.}
  \label{fig:tt-reconstruction-speedup}
\end{figure}

These timings cover reconstruction from precomputed cores on the measured A100.
They exclude loading, TT-SVD, rank selection, orthogonalization, quantization,
allocation, and transfers.

\subsection{Limitations}
The public study contains four complete hyperspectral cubes and two dense
FROSTT subtensors; it does not evaluate the full sparse FROSTT tensors. The
six-way tests are synthetic with mode size 12. Evaluation of the full sparse
tensors at their original scale would require sparse or randomized TT methods.
The accuracy experiments use dense TT-SVD and FP64 reconstruction.

\section{Fixed-Memory Hyperspectral Compression} \label{sec:hyperspectral_application}

\begin{sloppypar}
We compare standard TT baselines under fixed storage budgets on the four
hyperspectral tensors from \cref{sec:tt_experiments}; specialized
hyperspectral codecs are outside the scope of this experiment. Reconstruction
quality is evaluated after undoing the Frobenius normalization. We report
relative Frobenius error, peak signal-to-noise ratio (PSNR), a global per-band
structural-similarity (SSIM) variant based on Wang et
al.~\cite{WangBovikSheikhSimoncelli2004}, and spectral angle mapper
(SAM)~\cite{KruseEtAl1993SIPS}. PSNR and the SSIM variant are averaged over
spectral bands, and SAM over nonzero pixel spectra.

For each reference band \(b\), let
\(L_b=\max(X_b)-\min(X_b)\), with
\(L_b=\max(\max|X_b|,1)\) for a constant band. We average
\(\mathrm{PSNR}_b=20\log_{10}(L_b/\sqrt{\mathrm{MSE}_b})\) over bands.
The SSIM variant is evaluated once per complete band using population means,
variances, and covariance with \(C_1=(0.01L_b)^2\) and
\(C_2=(0.03L_b)^2\), and is then averaged over bands. SAM is the mean spectral
angle over pixels for which the product of the reference and reconstructed
spectrum norms is positive.

TT-SVD is performed in FP64 and the resulting cores are rounded before
reconstruction. The comparison therefore concerns representation quality at
fixed memory, not low-precision TT-SVD.

For each dataset we tested \(R\in\{4,8,16,32\}\) and
\(\delta\in\{1,2,4\}\), with TT ranks capped by the corresponding unfolding
dimensions. The compared methods are FP64 TT-SVD at rank \(R\), FP32 and FP16 rounded TT cores at the same rank, FP32 and FP16 rank-compensated cores at rank \(R+\delta\), and a memory-matched FP64 TT-SVD rank chosen as the largest FP64 rank whose exact TT storage does not exceed the low-precision rank-compensated storage. Compression ratios are relative to the dense FP64 tensor.

All 96 rank-compensated representations satisfy the certificate and strictly
reduce error. All 48 FP16 cases and 40 of 48 FP32 cases also reduce storage.
Although every strict improvement in this grid is certified, this empirical
converse is not guaranteed by the theorem. At each budget in
\cref{tab:hyperspectral_fixed_memory}, the FP16 rank-compensated representation
has the lowest relative error among the methods compared.
\end{sloppypar}

The repeated FP64 memory-matched rank \((17,17)\) in
\Cref{tab:hyperspectral_fixed_memory} does not come from one shared byte budget.
For each dataset, the budget is the dataset-specific storage of its FP16
rank-compensated \((36,36)\) representation. The reported, rounded budgets are
\(0.38\), \(0.89\), \(0.23\), and \(0.59\)~MiB for Indian Pines, PaviaU,
Salinas-A, and Salinas, respectively.
Exact TT storage counting makes \((17,17)\) the largest tested FP64 rank vector
within each of those four budgets.

\begin{figure}[!htbp]
  \centering
  \begin{minipage}{0.49\linewidth}
    \centering
    \includegraphics[width=\linewidth]{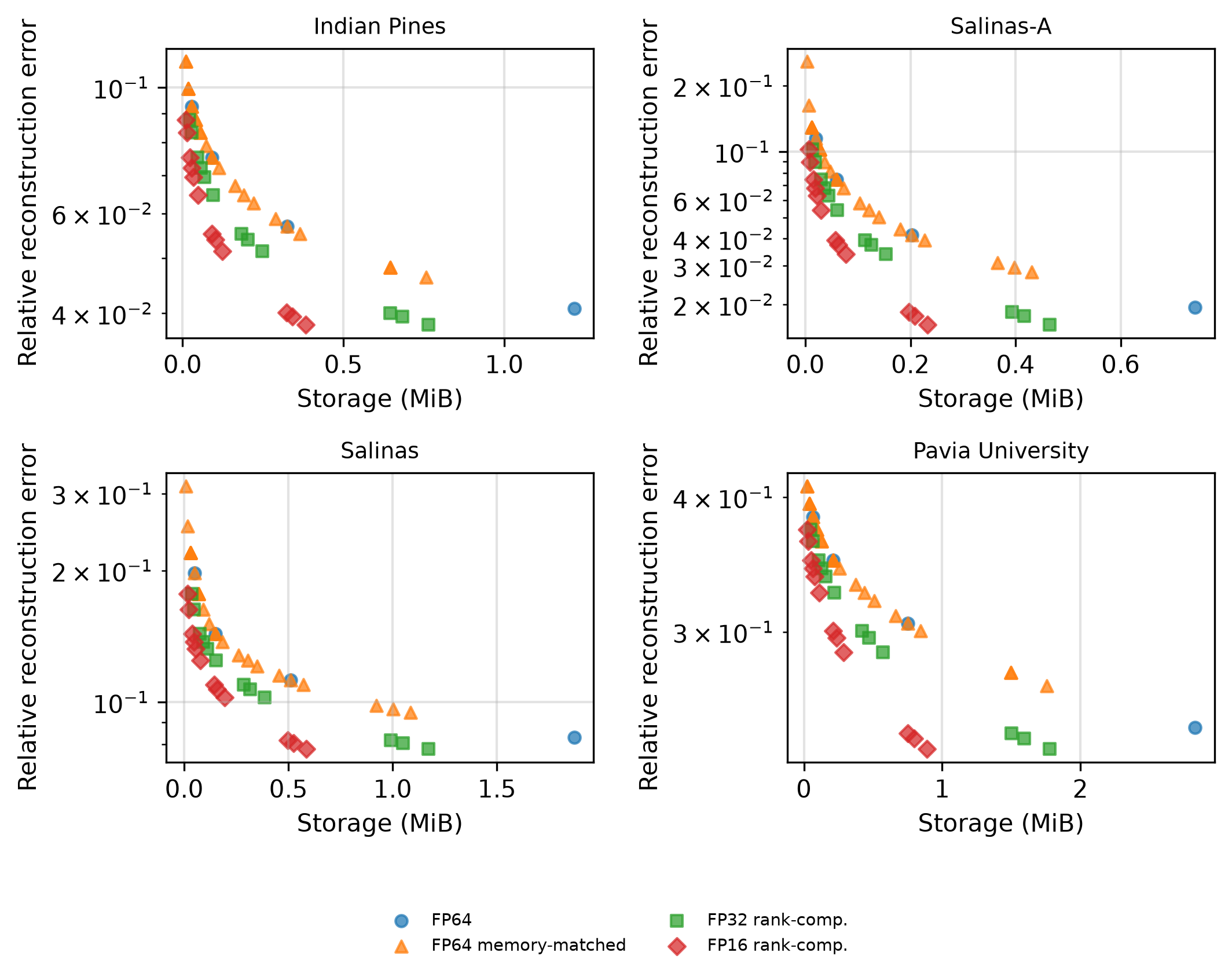}\\[-0.3ex]
    {\small (a) Relative error.}
  \end{minipage}\hfill
  \begin{minipage}{0.49\linewidth}
    \centering
    \includegraphics[width=\linewidth]{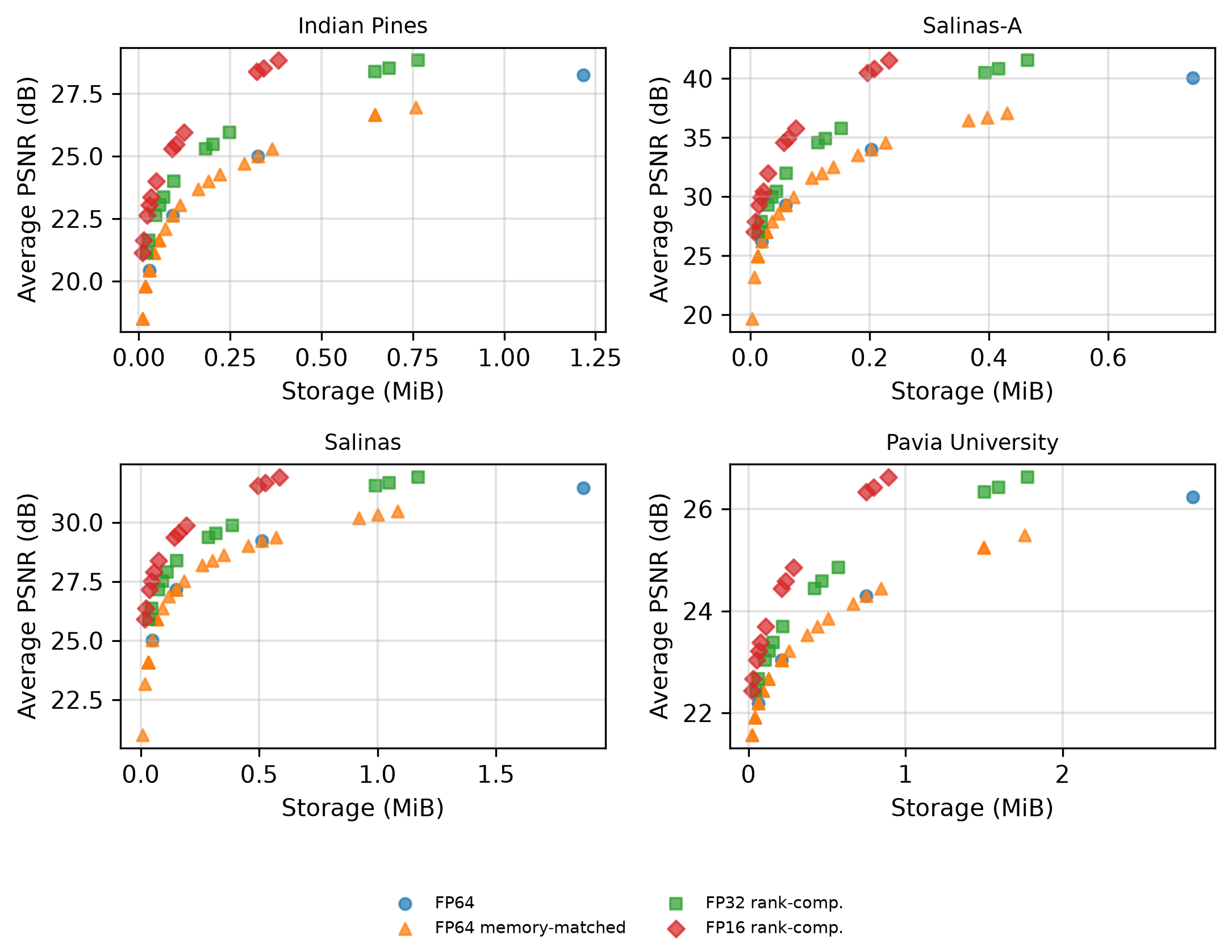}\\[-0.3ex]
    {\small (b) Average PSNR.}
  \end{minipage}
  \caption{Fixed-memory hyperspectral compression trade-off. The panels compare relative reconstruction error and average bandwise PSNR against storage for FP64 baselines, memory-matched FP64 models, and rank-compensated FP32/FP16 TT representations. In the tested grid, the low-precision rank-compensated models occupy lower-error and higher-PSNR regions at comparable or lower storage.}
  \label{fig:application_error_storage}
\end{figure}

The representative metrics in
\cref{tab:hyperspectral_compression_summary} show lower error and SAM and
higher PSNR and SSIM for both rank-compensated precisions on all four datasets;
FP16 uses half the storage of the corresponding FP32 representation.
\begin{table}[!htbp]
\centering
\caption{Representative hyperspectral TT reconstruction quality. The FP64 baseline uses rank $(32,32)$; the rank-compensated FP32 and FP16 models use rank $(36,36)$. PSNR and the global per-band SSIM variant are averaged over spectral bands, and SAM is in degrees.}
\label{tab:hyperspectral_compression_summary}
\resizebox{\textwidth}{!}{%
\begin{tabular}{llrrrrr}
\toprule
Dataset & Method & Storage (MiB) & Rel. error & PSNR & SSIM & SAM \\
\midrule
Indian Pines & FP64 baseline & 1.22 & 4.07e-02 & 28.24 & 0.934 & 1.87 \\
Indian Pines & FP32 rank-comp. & 0.76 & 3.82e-02 & 28.83 & 0.941 & 1.78 \\
Indian Pines & FP16 rank-comp. & 0.38 & 3.82e-02 & 28.83 & 0.941 & 1.78 \\
PaviaU & FP64 baseline & 2.83 & 2.45e-01 & 26.24 & 0.830 & 8.45 \\
PaviaU & FP32 rank-comp. & 1.78 & 2.34e-01 & 26.63 & 0.847 & 8.13 \\
PaviaU & FP16 rank-comp. & 0.89 & 2.34e-01 & 26.63 & 0.847 & 8.13 \\
Salinas-A & FP64 baseline & 0.74 & 1.93e-02 & 40.06 & 0.992 & 0.68 \\
Salinas-A & FP32 rank-comp. & 0.46 & 1.62e-02 & 41.55 & 0.994 & 0.58 \\
Salinas-A & FP16 rank-comp. & 0.23 & 1.62e-02 & 41.55 & 0.994 & 0.58 \\
Salinas & FP64 baseline & 1.87 & 8.30e-02 & 31.45 & 0.954 & 2.45 \\
Salinas & FP32 rank-comp. & 1.17 & 7.81e-02 & 31.92 & 0.959 & 2.35 \\
Salinas & FP16 rank-comp. & 0.59 & 7.81e-02 & 31.92 & 0.959 & 2.35 \\
\bottomrule
\end{tabular}}
\end{table}

\begin{table}[!htbp]
\centering
\caption{Fixed-memory comparison for hyperspectral TT compression. The dataset-specific budget is chosen from the best FP16 rank-compensated row for that dataset, and each competing method uses its best tested representation within that budget.}
\label{tab:hyperspectral_fixed_memory}
\resizebox{\textwidth}{!}{%
\begin{tabular}{lrrrrrrrl}
\toprule
Dataset & Budget (MiB) & FP64 rank & FP64 err. & FP32 rank & FP32 err. & FP16 rank & FP16 err. & Best under budget \\
\midrule
Indian Pines & 0.38 & $(17,17)$ & 5.53e-02 & $(20,20)$ & 5.15e-02 & $(36,36)$ & 3.82e-02 & FP16 rank-comp. \\
PaviaU & 0.89 & $(17,17)$ & 3.01e-01 & $(20,20)$ & 2.87e-01 & $(36,36)$ & 2.34e-01 & FP16 rank-comp. \\
Salinas-A & 0.23 & $(17,17)$ & 3.93e-02 & $(20,20)$ & 3.40e-02 & $(36,36)$ & 1.62e-02 & FP16 rank-comp. \\
Salinas & 0.59 & $(17,17)$ & 1.10e-01 & $(20,20)$ & 1.02e-01 & $(36,36)$ & 7.81e-02 & FP16 rank-comp. \\
\bottomrule
\end{tabular}}
\end{table}

Complete storage, certification, and reconstruction-quality values are reported
in \cref{tab:hyperspectral_storage_detail,tab:hyperspectral_quality_detail}.
\begin{table}[!htbp]
\centering
\caption{Storage and certification results for the hyperspectral TT representations. Comp. is the compression ratio relative to dense FP64 storage. The Prac. column is defined only for rank-compensated rows.}
\label{tab:hyperspectral_storage_detail}
\small
\setlength{\tabcolsep}{3.5pt}
\begin{tabularx}{\textwidth}{@{}l>{\raggedright\arraybackslash}Xllrrll@{}}
\toprule
Dataset & Method & Rank & Prec. & Storage (MiB) & Comp. & Cert. & Prac. \\
\midrule
Indian Pines & FP64 baseline & $(32,32)$ & FP64 & 1.22 & 26 & -- & -- \\
Indian Pines & FP32 same rank & $(32,32)$ & FP32 & 0.61 & 53 & -- & -- \\
Indian Pines & FP16 same rank & $(32,32)$ & FP16 & 0.30 & 105 & -- & -- \\
Indian Pines & FP32 rank-comp. & $(36,36)$ & FP32 & 0.76 & 42 & Yes & Yes \\
Indian Pines & FP16 rank-comp. & $(36,36)$ & FP16 & 0.38 & 84 & Yes & Yes \\
Indian Pines & FP64 mem.-matched & $(17,17)$ & FP64 & 0.36 & 88 & -- & -- \\
PaviaU & FP64 baseline & $(32,32)$ & FP64 & 2.83 & 58 & -- & -- \\
PaviaU & FP32 same rank & $(32,32)$ & FP32 & 1.42 & 115 & -- & -- \\
PaviaU & FP16 same rank & $(32,32)$ & FP16 & 0.71 & 230 & -- & -- \\
PaviaU & FP32 rank-comp. & $(36,36)$ & FP32 & 1.78 & 92 & Yes & Yes \\
PaviaU & FP16 rank-comp. & $(36,36)$ & FP16 & 0.89 & 183 & Yes & Yes \\
PaviaU & FP64 mem.-matched & $(17,17)$ & FP64 & 0.84 & 194 & -- & -- \\
Salinas-A & FP64 baseline & $(32,32)$ & FP64 & 0.74 & 15 & -- & -- \\
Salinas-A & FP32 same rank & $(32,32)$ & FP32 & 0.37 & 30 & -- & -- \\
Salinas-A & FP16 same rank & $(32,32)$ & FP16 & 0.19 & 60 & -- & -- \\
Salinas-A & FP32 rank-comp. & $(36,36)$ & FP32 & 0.46 & 24 & Yes & Yes \\
Salinas-A & FP16 rank-comp. & $(36,36)$ & FP16 & 0.23 & 48 & Yes & Yes \\
Salinas-A & FP64 mem.-matched & $(17,17)$ & FP64 & 0.23 & 49 & -- & -- \\
Salinas & FP64 baseline & $(32,32)$ & FP64 & 1.87 & 92 & -- & -- \\
Salinas & FP32 same rank & $(32,32)$ & FP32 & 0.94 & 185 & -- & -- \\
Salinas & FP16 same rank & $(32,32)$ & FP16 & 0.47 & 370 & -- & -- \\
Salinas & FP32 rank-comp. & $(36,36)$ & FP32 & 1.17 & 148 & Yes & Yes \\
Salinas & FP16 rank-comp. & $(36,36)$ & FP16 & 0.59 & 295 & Yes & Yes \\
Salinas & FP64 mem.-matched & $(17,17)$ & FP64 & 0.57 & 303 & -- & -- \\
\bottomrule
\end{tabularx}
\end{table}

\begin{table}[!htbp]
\centering
\caption{Reconstruction quality for the hyperspectral TT representations in \cref{tab:hyperspectral_storage_detail}. PSNR and the global per-band SSIM variant are averaged over spectral bands; SAM is reported in degrees.}
\label{tab:hyperspectral_quality_detail}
\small
\begin{tabular}{llrrrr}
\toprule
Dataset & Method & Rel. error & PSNR & SSIM & SAM \\
\midrule
Indian Pines & FP64 baseline & 4.07e-02 & 28.24 & 0.934 & 1.87 \\
Indian Pines & FP32 same rank & 4.07e-02 & 28.24 & 0.934 & 1.87 \\
Indian Pines & FP16 same rank & 4.07e-02 & 28.24 & 0.934 & 1.87 \\
Indian Pines & FP32 rank-comp. & 3.82e-02 & 28.83 & 0.941 & 1.78 \\
Indian Pines & FP16 rank-comp. & 3.82e-02 & 28.83 & 0.941 & 1.78 \\
Indian Pines & FP64 mem.-matched & 5.53e-02 & 25.29 & 0.885 & 2.39 \\
PaviaU & FP64 baseline & 2.45e-01 & 26.24 & 0.830 & 8.45 \\
PaviaU & FP32 same rank & 2.45e-01 & 26.24 & 0.830 & 8.45 \\
PaviaU & FP16 same rank & 2.45e-01 & 26.24 & 0.830 & 8.45 \\
PaviaU & FP32 rank-comp. & 2.34e-01 & 26.63 & 0.847 & 8.13 \\
PaviaU & FP16 rank-comp. & 2.34e-01 & 26.63 & 0.847 & 8.13 \\
PaviaU & FP64 mem.-matched & 3.01e-01 & 24.45 & 0.717 & 10.10 \\
Salinas-A & FP64 baseline & 1.93e-02 & 40.06 & 0.992 & 0.68 \\
Salinas-A & FP32 same rank & 1.93e-02 & 40.06 & 0.992 & 0.68 \\
Salinas-A & FP16 same rank & 1.93e-02 & 40.06 & 0.992 & 0.68 \\
Salinas-A & FP32 rank-comp. & 1.62e-02 & 41.55 & 0.994 & 0.58 \\
Salinas-A & FP16 rank-comp. & 1.62e-02 & 41.55 & 0.994 & 0.58 \\
Salinas-A & FP64 mem.-matched & 3.93e-02 & 34.60 & 0.977 & 1.24 \\
Salinas & FP64 baseline & 8.30e-02 & 31.45 & 0.954 & 2.45 \\
Salinas & FP32 same rank & 8.30e-02 & 31.45 & 0.954 & 2.45 \\
Salinas & FP16 same rank & 8.30e-02 & 31.45 & 0.954 & 2.45 \\
Salinas & FP32 rank-comp. & 7.81e-02 & 31.92 & 0.959 & 2.35 \\
Salinas & FP16 rank-comp. & 7.81e-02 & 31.92 & 0.959 & 2.35 \\
Salinas & FP64 mem.-matched & 1.10e-01 & 29.36 & 0.924 & 2.99 \\
\bottomrule
\end{tabular}
\end{table}

\FloatBarrier

\section{Conclusion} \label{sec:conclusion}

For matrices, the certificate guarantees that a lower-precision rank-$(k+1)$
approximation is no worse than the FP64 rank-$k$ baseline when the reduction in
truncation error offsets the storage perturbation. Public SuiteSparse
experiments support this mechanism and show FP16 failures near the perturbation
floor. For resident matrix application at batch 8192, the
rank-compensated FP32 and FP16 factors attain $1.28\times$ and $2.12\times$
geometric-mean speedups on the measured A100, while neither accelerates at
batch 256. Thus the matrix storage--accuracy benefit translates into downstream
speed only when the workload amortizes fixed overhead.

The TT extension preserves the matrix structure but is conditional and a
posteriori: the realized TT truncation gain and rounded-core perturbation must be
measured for the computed representation. The three-way and six-way synthetic
tests and the public hyperspectral and FROSTT tests support this diagnostic,
while also showing the FP16 perturbation-floor failures and the need for a
separate memory check. Reconstruction of the six public tensors gives
$1.38\times$ and $1.94\times$ geometric-mean speedups for rank-compensated FP32
and FP16, respectively, on the same A100 configuration. These measurements do
not establish low-precision TT-SVD acceleration.

Under the tested fixed-memory budgets, the rank-compensated hyperspectral models
improve reconstruction quality. The evidence is limited to the reported
matrices, synthetic tensors, four hyperspectral cubes, two FROSTT subtensors,
and one GPU configuration. Future work includes full sparse tensors, additional
hardware, and end-to-end low-precision decomposition.

\FloatBarrier

\section*{Data and Code Availability}

The source code, configurations, and processed results are available in the
\href{https://github.com/ceciver/rank-compensation-repro}{GitHub repository} and
archived as version 1.0.0 on Zenodo~\cite{ElMountasserJbilou2026Code}. The
public datasets are available from the sources cited above.

\bibliographystyle{siamplain}
\bibliography{references}

\end{document}